\documentclass{article}
\usepackage{arxiv}
\usepackage{graphicx} 
\usepackage{amsmath}  
\usepackage{amssymb}  
\usepackage{xcolor}   
\usepackage{geometry} 

\usepackage[utf8]{inputenc} 
\usepackage[T1]{fontenc}    
\usepackage{url}            
\usepackage{booktabs}       
\usepackage{amsfonts}       
\usepackage{nicefrac}       
\usepackage{microtype}      
\usepackage{lipsum}
\graphicspath{ {./img/} }
\usepackage[toc,page]{appendix}

\usepackage{times}
\usepackage{latexsym}
\usepackage{xspace}
\usepackage{multirow}
\usepackage{fontawesome5}
\usepackage{amsthm}
\usepackage{tikz}
\usetikzlibrary{graphs}
\usepackage{listings}
\usepackage{amssymb}
\usepackage{array} 
\usepackage[pagebackref]{hyperref}
\usepackage{bm}
\usepackage{inconsolata}
\usepackage{algorithm,algorithmic}

\newtheorem{theorem}{Theorem}[section]
\newtheorem{lemma}[theorem]{Lemma}
\newtheorem{corollary}[theorem]{Corollary}

\theoremstyle{proposition}
\newtheorem{proposition} [theorem]{Proposition}
\theoremstyle{definition}  
\newtheorem{definition} [theorem] {Definition} 

\newtheorem{example} [theorem] {Example}

\DeclareMathOperator{\val}{val}
\DeclareMathOperator{\outdeg}{outdeg}
\DeclareMathOperator{\Div}{Div}
\DeclareMathOperator{\Prin}{Prin}
\DeclareMathOperator{\Pic}{Pic}
\DeclareMathOperator{\Jac}{Jac}

\theoremstyle{definition}  
\theoremstyle{proposition}

\usepackage{tcolorbox}
\tcbuselibrary{breakable,xparse,skins}

\usepackage{listings}
\usepackage{xcolor}
\title{\texttt{chipfiring}: A Python Package for Efficient Mathematical Analysis of Chip-Firing Games on MultiGraphs}

\author{
 Dhyey Dharmendrakumar Mavani \\
  Departments of Mathematics, \\ 
  Statistics and Computer Science \\
  Amherst College \\
  Amherst, MA, 01002 \\
  \texttt{ddmavani2003@gmail.com} \\
   \And
Tairan (Ryan) Ji \\
  Department of Computer Science \\
  Amherst College\\
  Amherst, MA 01002 \\
  \texttt{tji26@amherst.edu} \\
  \And
   Nathan Pflueger \\
  Department of Mathematics\\
  Amherst College\\
  Amherst, MA 01002 \\
  \texttt{npflueger@amherst.edu} \\
}

\begin{document}
\maketitle
\begin{abstract}
This paper presents \texttt{chipfiring}, a comprehensive Python package for the mathematical analysis of chip-firing games on finite graphs. The package provides a robust toolkit for defining graphs and chip configurations (divisors), performing chip-firing operations, and analyzing fundamental properties such as winnability, linear equivalence, and divisor rank. We detail the core components of the library, including its object-oriented graph and divisor implementations, integrated Laplacian matrix computations, and an efficient implementation of Dhar's algorithm for determining the solvability of the dollar game. The \texttt{chipfiring} package is designed for researchers and students in graph theory, combinatorics, and algebraic geometry, providing essential algorithms and data structures for exploring these rich mathematical models. We describe the library's architecture, illustrate its usage with comprehensive examples, and highlight its specialized contributions compared to general-purpose graph libraries.
\end{abstract}


\section{Introduction}
The chip-firing game, also known as the dollar game or sandpile model, is a discrete dynamical system on a graph that has found diverse applications across mathematics and physics, from modeling self-organized criticality \cite{dhar1990self} to providing insights into the structure of algebraic curves \cite{baker2007riemann}. In its most common formulation, vertices on a graph are endowed with a number of chips (which can be negative, representing debt), and a vertex can "fire" by sending a chip to each of its neighbors. A central problem is to determine if a given configuration of chips, or \textit{divisor}, can reach a state with no vertices in debt through a sequence of legal firing moves.

Despite the model's importance, a dedicated, user-friendly software package for its study in Python has been lacking. While general-purpose graph libraries exist, they do not provide the specialized data structures, operations, and algorithms central to chip-firing theory. The \texttt{chipfiring} package aims to fill this gap.

This paper describes our Python-based scientific software package for setting up and analyzing chip-firing games. We focus on the tools for constructing graphs and divisors, performing firing moves, and solving the winnability problem. The package provides a clean, object-oriented API with comprehensive documentation and type hints, following modern Python standards. The package also includes interactive visualizations for selected functions to help mathematicians develop intuition and assist students in learning key concepts. The package is available on PyPI, detailed documentation is deployed via ReadTheDocs, and the source code is available on the project's GitHub repository.

\begin{itemize}
    \item PyPI: \url{https://pypi.org/project/chipfiring}
    \item Documentation: \url{https://chipfiring.readthedocs.io}
    \item GitHub: \url{https://github.com/DhyeyMavani2003/chipfiring}
\end{itemize}

Since the package is readily available on the Python Package Index (PyPI), it can be installed by running:

\begin{lstlisting}[language=Python]
pip install chipfiring
\end{lstlisting}

For the rest of this paper, we will refer to version 1.1.1 of the package.

The package has minimal dependencies, requiring only NumPy for numerical computations, NetworkX for specific gonality-based algorithms, and Dash, Dash Cytoscape, and Dash Bootstrap Components for visualization capabilities. This lightweight approach ensures easy integration into existing mathematical computing environments with official support for Python versions 3.8, 3.9, 3.10, 3.11, 3.12, and 3.13, as confirmed by intermittent GitHub-Actions runners on our GitHub repository.

The \texttt{chipfiring} package is a specialized tool. While libraries like NetworkX \cite{networkx} are excellent for general graph analysis, they lack the domain-specific features for chip-firing. Table \ref{tab:comparison} provides a feature comparison.

\begin{table}[h]
 \caption{Feature comparison with general graph libraries.}
  \centering
  \begin{tabular}{lcc}
    \toprule
    \textbf{Feature} & \textbf{\texttt{chipfiring}} & \textbf{General Library (e.g., NetworkX)} \\
    \midrule
    Graph Representation & Yes & Yes \\
    Divisor/State Object & Yes & No \\
    Chip-Firing Operations & Yes & No \\
    Dhar's Algorithm & Yes & No \\
    Divisor Rank/Reduction & Yes & No \\
    Linear Equivalence Testing & Yes & No \\
    Q-Reduction Algorithm & Yes & No \\
    \bottomrule
  \end{tabular}
  \label{tab:comparison}
\end{table}

In the following sections, we will review the mathematical background, which is also summarized in \cite{mavani2025chipfiringthesis}. We will provide embedded commentaries and demonstrations on using the \texttt{chipfiring} package. Comprehensive discussion of the functions and arguments within our API framework may be found at  \url{https://chipfiring.readthedocs.io}.

The organization and notation of this paper closely follows the textbook of Corry and Perkinson \cite{corry2018divisors}, and our package consists in large part of implementing the algorithms and formalizing the abstractions therein. 
We have also drawn inspiration from the excellent exposition on chip-firing and graph gonality in \cite{beougher2023chip}, and have included some premade constructions for the reader to experiment with the graphs studied there.
Indeed, a reader new to this subject would be well served to study Corry and Perkinson's excellent textbook, or the paper \cite{beougher2023chip}, side-by-side with this paper and a Python notebook, trying experiments and examples of each new idea using the \texttt{chipfiring} package. We hope very much that some readers, who will undoubtedly encounter functionality they wish were present or possible improvements to the implementations, will contribute to the project; see Section \ref{sec:community}.

\section{Creating and visualizing graphs}
\label{sec:graphs}

Before proceeding to the details of the chipfiring game, we demonstrate the syntax for creating (multi-) graphs in the \texttt{chipfiring} package. Graphs are created by specifying a set of vertex names, represented as strings, and a list of edges with specified multiplicities. Once created, a graph can be visualy displayed using the \texttt{visualize} command. Here is a first example.

\begin{lstlisting}[language=Python]
from chipfiring import CFGraph, CFDivisor, visualize

# Define a graph
vertices = {"Alice", "Bob", "Charlie", "Elise"}
edges = [
("Alice", "Bob", 1), ("Alice", "Charlie", 1), ("Alice", "Elise", 2),
("Bob", "Charlie", 1), ("Charlie", "Elise", 1)
]
graph = CFGraph(vertices, edges)

# Note that this launches a web-based interactive visualization at localhost:8050,
# so the program will halt until the user manually closes the visualization
# (Ctrl + C in the terminal)
visualize(graph)
\end{lstlisting}

The resulting visualization is shown in Figure \ref{fig:visualization}.

\begin{figure}
    \centering
    \includegraphics[width=0.6\linewidth]{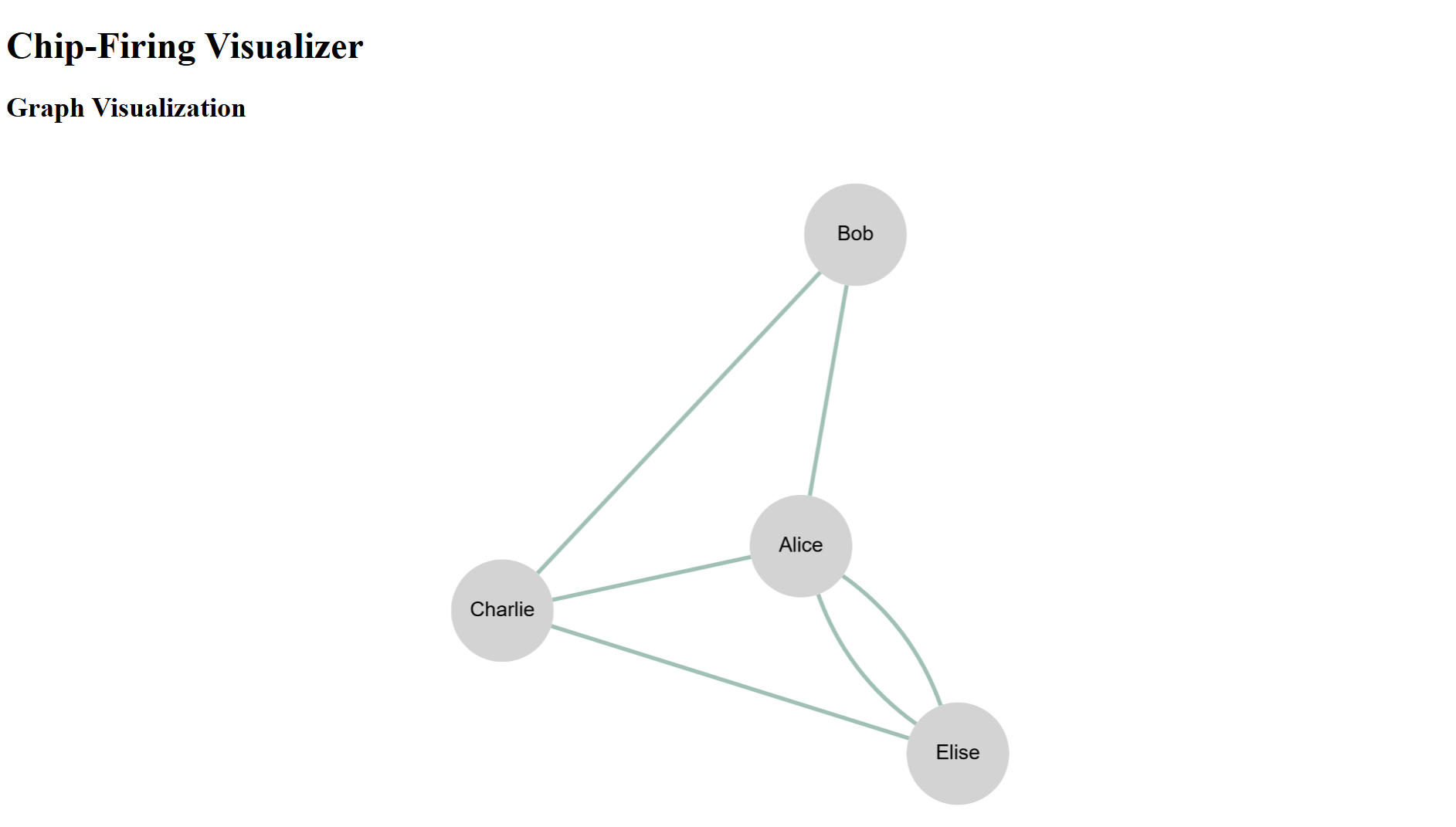} 
    \includegraphics[width=0.3\linewidth]{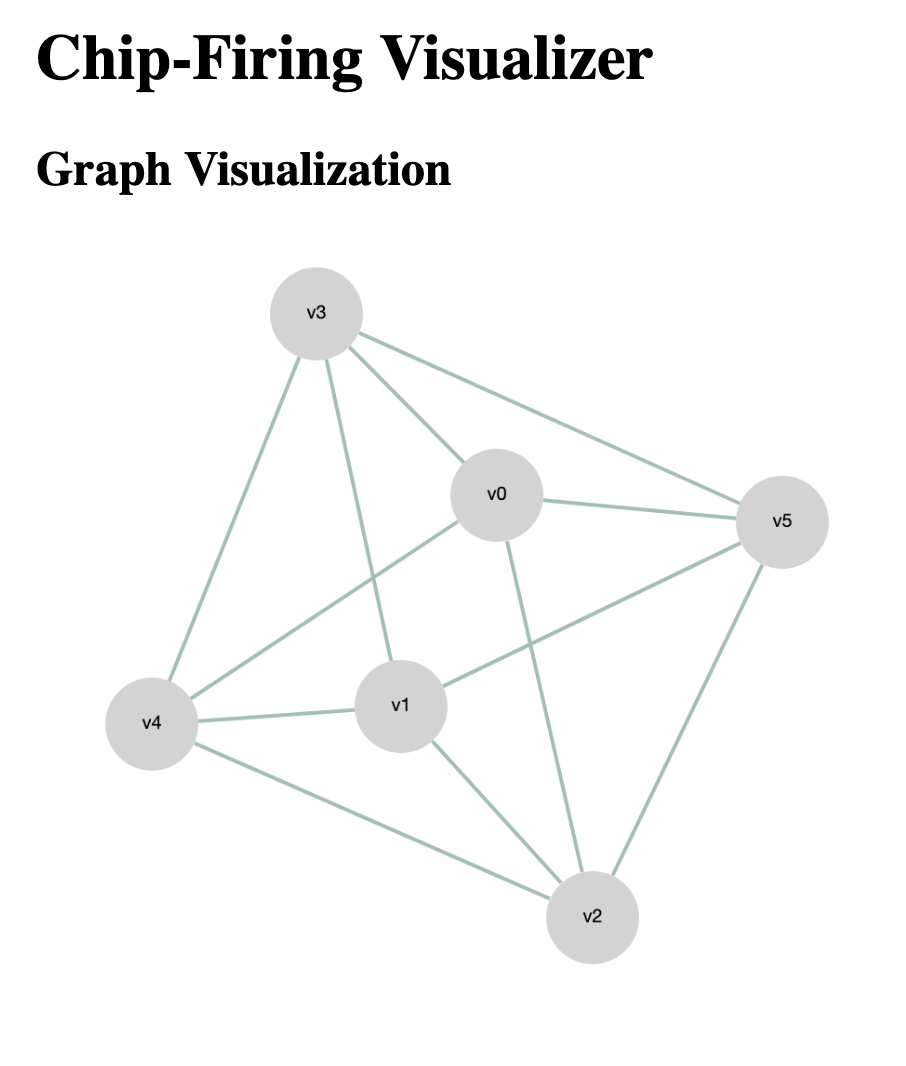}
    \caption{Screenshots of two example CFGraph visualizations.}
    \label{fig:visualization}
\end{figure}

The package also provides several pre-built graph constructions studied in \cite{beougher2023chip}, which is an excellent introduction to the study on graph gonality. The code below shows how to create these graphs, and visualize the octahedron graph. The output is shown in Figure \ref{fig:visualization}.

\begin{lstlisting}[language=Python]
from chipfiring import tetrahedron, cube, octahedron,  dodecahedron, icosahedron, visualize

# Platonic solid graphs (available for research purposes)
tetra = tetrahedron()  # K_4
cube_graph = cube()    # 3-regular, 8 vertices
octa = octahedron()    # Complete tripartite K_{2,2,2}
dodec = dodecahedron()  # 3-regular, 20 vertices
icos = icosahedron() # 5-regular, 12 vertices

# For example, we can visualize the octahedron
visualize(octa)
\end{lstlisting}

\section{The Dollar Game}
\label{ch:dollar-game}

Consider a graph $G = (V, E)$ with $V$ as a set of vertices representing people and $E$ as a set of edges representing relationships between them. The more edges between individuals, the stronger the relationship. At each vertex, we also record their wealth via an integer representing the number of dollars they have, with negative values indicating debt. The goal is to find a sequence of lending/borrowing moves after which everyone becomes debt-free. In one move, a vertex can lend money (\textit{fire}) or borrow money by taking or sending 1 unit of currency across each edge it shares in the graph. This is called the \textit{dollar game on $G$}, and if such a \textbf{sequence exists}, the game is said to be \textbf{winnable}.

Let us walk through an example in Figure \ref{fig:div-setup} to illustrate the setup better.

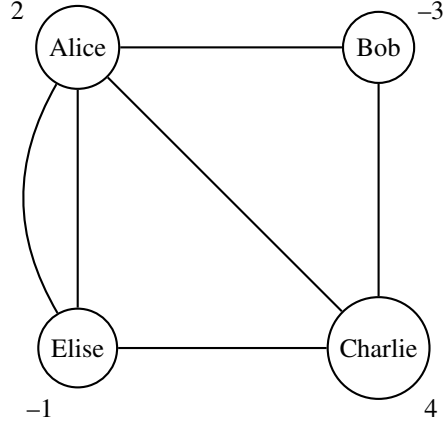
\begin{figure}[h]
    \centering
    \begin{tikzpicture}[scale=1, transform shape, thick]
      \node[circle,draw] (Alice) at (0,4) {Alice};
      \node[circle,draw] (Bob) at (4,4) {Bob};
      \node[circle,draw] (Elise) at (0,0) {Elise};
      \node[circle,draw] (Charlie) at (4,0) {Charlie};
    
      \draw (Alice) -- (Bob);
      \draw (Alice) -- (Elise);
      \draw (Alice) -- (Charlie);
      \draw (Bob) -- (Charlie);
      \draw (Elise) -- (Charlie);
      \draw (Elise) to[bend left] (Alice);
    
      \node at (-0.8,4.5) {2};
      \node at (4.7,4.5) {--3};
      \node at (-0.5,-0.8) {--1};
      \node at (4.7,-0.8) {4};
    
    \end{tikzpicture}
    \caption{Situational wealth distribution \& relationship setup.}
    \label{fig:div-setup}
\end{figure}

Alice starts with a wealth of $2$, Bob is in debt with $-3$, Charlie has $4$, and Elise is slightly in debt with $-1$. The edges between vertices represent relationships; in each turn, a person can either lend, borrow, or do nothing with \textbf{all} the edges they are connected to. For instance, Charlie can lend $1$ to each of Elise, Alice, and Bob, which takes Elise out of debt, leaving Alice with $2$ and Bob with $-2$. The game continues until no one remains in debt, and if such a redistribution is possible, the game is considered \textbf{winnable}.

Before exploring solutions to this game, let us formalize this setup. Our discussion closely follows the exposition of \cite{corry2018divisors}.

\subsection{Divisors \& Linear Equivalence}
\label{def_G}
When we mention a \textbf{graph}, we refer to a \textit{finite, connected, undirected multigraph without loop edges}. Essentially, a \textbf{multigraph} $G = (V, E)$ consists of two components: a set of vertices $V$ and a multiset of edges $E$, where each edge is an unordered pair $\{v, w\}$ representing connections between vertices. The prefix ``multi'' means that pairs like $\{v, w\}$ can appear multiple times in $E$. We often simplify notation by writing an edge as $vw$. A multigraph $G$ is considered finite if both $V$ and $E$ are finite and connected if there is a path of edges between any two vertices.

\definition (cf. \cite[Definition 1.3]{corry2018divisors}) \label{def_DivG} A \textbf{divisor} on the graph $G$ is an element of the \textit{free abelian group} on its vertices:
\[
\Div(G) = \mathbb{Z}^V = \left\{ \sum_{v \in V} D(v) v \colon D(v) \in \mathbb{Z} \right\}.
\]

Divisors can be thought of as ways to describe the wealth distribution on $G$. If $D = \sum_{v \in V} D(v) \cdot v \in \Div(G)$, then $D(v)$ gives the amount of money at the vertex (or person) $v$, with negative values representing debt. The total money in the system is captured by the \textit{degree} of the divisor as defined below.

\definition (cf.  \cite[Definition 1.4]{corry2018divisors}) \label{def_degD} The \textbf{degree} of a divisor $D = \sum_{v \in V} D(v) \cdot v \in \Div(G)$ is defined as:
\[
\deg(D) = \sum_{v \in V} D(v).
\]

For instance, from the example presented earlier in Figure \ref{fig:div-setup}, the divisor D can be represented as $D = 2(A) - 3(B) + 4(C) - (E)$, and thus the $\deg(D) = 2 - 3 + 4 - 1 = 2$.

We use $\Div^k(G)$ \label{def_DivkG} to denote all divisors with degree $k$, and $\Div_+(G)$ \label{def_Div+G} for divisors with a non-negative degree. Note that the word ``degree'' can refer to two things in our sub-domain at the intersection of graph theory and combinatorics, so for clarity, we write $\val(v)$ (valence of v) \label{def_val} for the number of edges connected to $v$.

In the package, multigraphs and divisors are represented by the \texttt{CFGraph} and \texttt{CFDivisor} classes, which are specifically designed for chip-firing operations. The setup for the example above can be reproduced as follows:

\begin{lstlisting}[language=Python]
from chipfiring import CFGraph, CFDivisor

# Define a graph
vertices = {"Alice", "Bob", "Charlie", "Elise"}
edges = [
("Alice", "Bob", 1), ("Alice", "Charlie", 1), ("Alice", "Elise", 2),
("Bob", "Charlie", 1), ("Charlie", "Elise", 1)
]
graph = CFGraph(vertices, edges)


# Define a divisor
divisor = CFDivisor(graph, [("Alice", 2), ("Bob", -3), ("Charlie", 4), ("Elise", -1)])
print(f"Total degree: {divisor.get_total_degree()}") # Output: 2
\end{lstlisting}

In addition to defining graphs and divisors directly, the package also provides comprehensive tools for importing and exporting to and from various formats; this functionality supports \texttt{CFGraph} and \texttt{CFDivisor} as well as \texttt{CFOrientation} and \texttt{CFiringscript}, which will be discussed later. Please refer to \autoref{appendix:import}, and accordingly format the files based on the I/O format of interest. Below we present an API code example to illustrate the same:

\begin{lstlisting}[language=Python]
from chipfiring import CFDataProcessor, visualize

# Initialize the data processor
processor = CFDataProcessor()

# Import graph, divisor, orientation, or firing script from json or txt
divisor = processor.read_json("divisor_input.json", "divisor")
graph = processor.read_txt("graph_input.txt", "graph")

# Export to LaTeX, json, or txt
processor.to_tex(graph, "graph_output.tex")  # LaTeX/TikZ export
processor.to_json(divisor, "divisor_data.json")  # JSON export
processor.to_txt(divisor, "divisor_data.txt")  # txt export
\end{lstlisting}

The package also provides interactive visualizations for the \texttt{CFGraph}, \texttt{CFDivisor}, and \texttt{CFOrientation} classes. An example of an \texttt{CFGraph} visualization, as generated by the code-block below, can be seen in \autoref{fig:divisorVisualization}.

\begin{lstlisting}[language=Python]
# Visualize a divisor
# As for graphs, this launches a web-based interactive visualization at localhost:8050,
# so the program will halt until the user manually closes the visualization
# (Ctrl + C in the terminal)
visualize(divisor)
\end{lstlisting}

\begin{figure}
    \centering
    \includegraphics[width=0.4\linewidth]{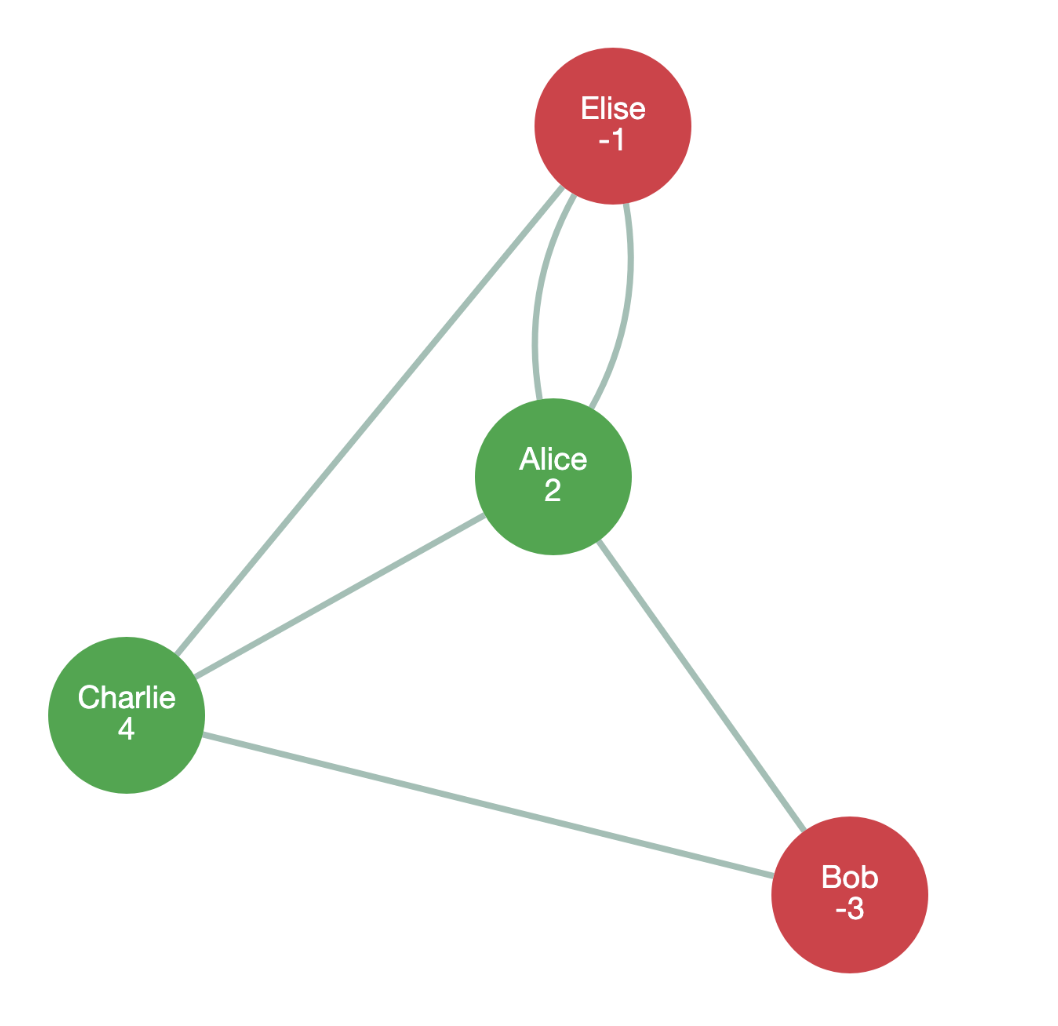}
    \caption{Screenshot of an CFDivisor visualization.}
    \label{fig:divisorVisualization}
\end{figure}

Now, let us define lending and borrowing moves for the chip-firing game. 

\definition (cf. \cite[Definition 1.5]{corry2018divisors}) \label{def-firing-move} Given divisors $D, D' \in \Div(G)$ and a vertex $v \in V$, we say $D'$ is obtained from $D$ by a \textit{lending move} at $v$, written as $D \overset{v}{\rightarrow} D'$, if:
\[
D' = D - \sum_{vw \in E} (v - w) = D - \val(v) \cdot v + \sum_{vw \in E} w.
\]
Similarly, $D'$ is obtained from $D$ by a \textit{borrowing move} at $v$, written as $D \overset{v}{\leftarrow} D'$, if:
\[
D' = D + \sum_{vw \in E} (v - w) = D + \val(v) \cdot v - \sum_{vw \in E} w.
\]

What is interesting here is that the order in which lending or borrowing happens does not matter. This gives the dollar game an \textbf{abelian property}, meaning the operations commute with each other.

\definition (cf. \cite[Definition 1.6]{corry2018divisors}) \label{def_setfiring} Suppose $D'$ is obtained from $D$ by lending from all the vertices in some subset $W \subseteq V$. In this case, we call this a \textit{set-lending} (or \textit{set-firing}) move by $W$, denoted $D \overset{W}{\longrightarrow} D'$.

Using these definitions, let us try to perform set-firing on the example divisor shown in Figure \ref{fig:div-setup} earlier in this section.

\begin{figure}[h]
    \centering
    \begin{tikzpicture}[scale=1, transform shape, thick]
      \node[circle,draw] (A1) at (0,4) {Alice};
      \node[circle,draw] (B1) at (2,4) {Bob};
      \node[circle,draw] (E1) at (0,2) {Elise};
      \node[circle,draw] (C1) at (2,2) {Charlie};
      \draw (A1)--(B1)--(C1)--(E1)--(A1);
      \draw (A1)--(C1);
      \draw (E1) to[bend left] (A1);
      \node at (-0.7,4.5) {2};
      \node at (2.6,4.5) {--3};
      \node at (-0.7,1.5) {--1};
      \node at (2.6,1.2) {4};
      \draw[red,thick,rounded corners] (-1,4.8)--(1.2,4.8)--(1.2,2.9)--(2.8,2.9)--(2.8,1)--(-1,1)--cycle;
      \node[red] at (0,5.2) {$W_1$ \small firing set};

      \draw[->,blue] (A1)--(B1);
      \draw[->,blue] (C1)--(B1);

      \node[circle,draw] (A2) at (6,4) {Alice};
      \node[circle,draw] (B2) at (8,4) {Bob};
      \node[circle,draw] (E2) at (6,2) {Elise};
      \node[circle,draw] (C2) at (8,2) {Charlie};
      \draw (A2)--(B2)--(C2)--(E2)--(A2);
      \draw (A2)--(C2);
      \draw (E2) to[bend left] (A2);
      \node at (5.3,4.5) {1};
      \node at (8.6,4.5) {--1};
      \node at (5.3,1.5) {--1};
      \node at (8.6,1.2) {3};
      \draw[red,thick,rounded corners] (5,4.8)--(7.2,4.8)--(7.2,2.9)--(8.8,2.9)--(8.8,1)--(5,1)--cycle;
      \node[red] at (6,5.2) {$W_2$ \small firing set};

      \draw[->,blue] (A2)--(B2);
      \draw[->,blue] (C2)--(B2);
      
      \node[circle,draw] (A3) at (0,0) {Alice};
      \node[circle,draw] (B3) at (2,0) {Bob};
      \node[circle,draw] (E3) at (0,-2) {Elise};
      \node[circle,draw] (C3) at (2,-2) {Charlie};
      \draw (A3)--(B3)--(C3)--(E3)--(A3);
      \draw (A3)--(C3);
      \draw (E3) to[bend left] (A3);
      \node at (-0.7,-0.5) {0};
      \node at (2.5,-0.5) {1};
      \node at (-0.7,-2.5) {--1};
      \node at (2.5,-2.8) {2};
      \draw[red,thick,rounded corners] (1.2,0.7)--(2.8,0.7)--(2.8,-3.1)--(1.2,-3.1)--cycle;
      \node[red] at (3,-3.4) {$W_3$ \small firing set};
      
      \draw[->,blue,thick] (B3)--(A3);
      \draw[->,blue,thick] (C3)--(A3);
      \draw[->,blue,thick] (C3)--(E3);
      
      \node[circle,draw] (A4) at (6,0) {Alice};
      \node[circle,draw] (B4) at (8,0) {Bob};
      \node[circle,draw] (E4) at (6,-2) {Elise};
      \node[circle,draw] (C4) at (8,-2) {Charlie};
      \draw (A4)--(B4)--(C4)--(E4)--(A4);
      \draw (A4)--(C4);
      \draw (E4) to[bend left] (A4);
      \node at (5.3,-0.5) {2};
      \node at (8.5,-0.5) {0};
      \node at (5.3,-2.5) {0};
      \node at (8.5,-2.8) {0};
      \node[blue, thick] at (9,-1) {\large WIN!!};

      \draw[->,thick] (3,3)--(5,3) node[midway,above] {$W_1$};
      \draw[->,thick] (5,1)--(3.0,0.5) node[midway,above] {$W_2$};
      \draw[->,thick] (3,-1)--(5,-1) node[midway,above] {$W_3$};

    \end{tikzpicture}
    \caption{Application of set-firing moves leading to a win in the case of the divisor mentioned in Figure \ref{fig:div-setup}}
    \label{fig:sf-walkthrough}
\end{figure}
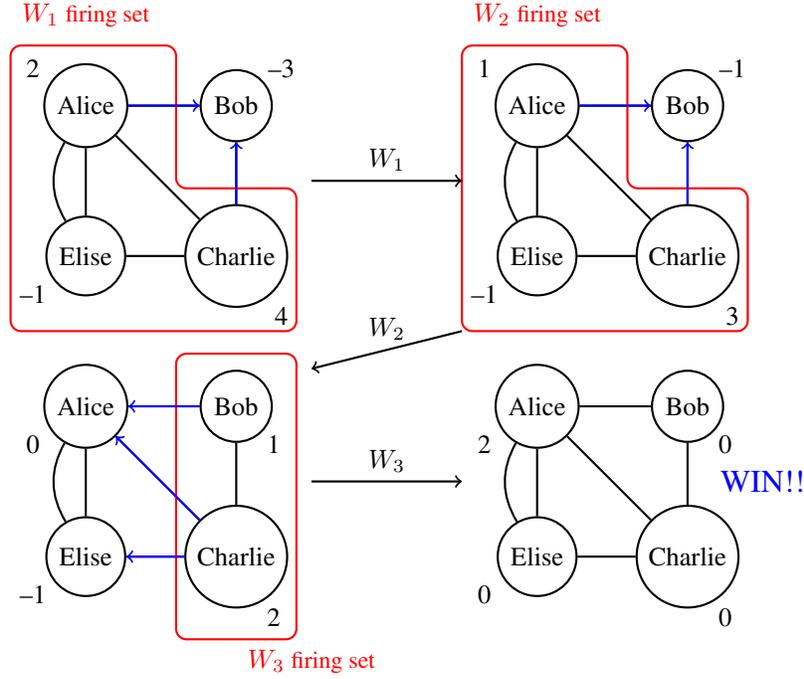

As shown in Figure \ref{fig:sf-walkthrough}, we can have a firing-set $W_1 = \{A, E, C\}$. After this firing move, all the internal lending and borrowing between the members of the firing set cancels out, and effective lending of $2$ happens to $B$ from $A$ \& $C$ lending $1$ each. Now, to get $B$ out of debt, we repeat the same set-firing move on this newly obtained divisor with $W_2 = W_1$. This gives us the divisor that can be represented as $0(A)+1(B)+2(C)-1(E)$. Finally, to get $E$ out of debt, we can carefully engineer our firing set to be $W_3=\{B, C\}$. This ensures we take the minimum number of lenders out of debt (vaguely speaking). After this move, as shown in the figure, all the graph members come out of debt, signifying that we have won the game!

The above steps can be performed via \texttt{chipfiring} as follows:
\begin{lstlisting}[language=Python]
# First firing: Alice, Elise, Charlie set fire
divisor.set_fire({"Alice", "Elise", "Charlie"})

# Second firing: Alice, Elise, Charlie set fire (again)
divisor.set_fire({"Alice", "Elise", "Charlie"})

# Third firing: Bob, Charlie set fire
divisor.set_fire({"Bob", "Charlie"})
\end{lstlisting}

\begin{proposition} (cf. \cite[Exercise 1.7]{corry2018divisors}) \label{pr:borrow-len-duality} 
Borrowing from a vertex \(v \in V\) is just like lending from all vertices in $V \setminus {v}$, and if we perform set-lending from all vertices in $V$, the net effect is zero. 
\end{proposition} 

\definition (cf. \cite[Definition 1.8]{corry2018divisors}) \label{def_lineq} A divisor $D$ is said to be \textbf{linearly equivalent} to another divisor $D'$, denoted $D \sim D'$, if we can obtain $D'$ from $D$ by a sequence of lending moves.

\definition (cf. \cite[Definition 1.11]{corry2018divisors}) \label{def_divclass} The \textbf{divisor class} determined by $D \in \Div(G)$ is:
\[
[D] = \{D' \in \Div(G) \colon D' \sim D\}.
\]

One can think of a divisor class as a self-contained economy where the total wealth does not change, but the distribution of wealth might shift around. In simpler terms, it represents all possible money distributions that can be achieved through lending.

\definition (cf.  \cite[Definition 1.13]{corry2018divisors}) \label{effective-div} A divisor $D$ is \textbf{effective} if $D(v) \geq 0$ for all $v \in V$, meaning no one is in debt. The set of effective divisors on $G$ is denoted by $\Div_+(G)$. \footnote{Since the set $\Div_+(G)$ does not have inverses, it is NOT a subgroup of $\Div(G)$, but rather a \textit{commutative monoid}.} We write this as $D\geq0$.

Thus, we can see that using the above definitions, the \textbf{objective} of the dollar game becomes: \textit{Is a given divisor linearly equivalent to an effective divisor?}

\definition (cf. \cite[Definition 1.14]{corry2018divisors}) \label{winnable-effective-div} A divisor $D$ is \textbf{winnable} if D is linearly equivalent to an effective divisor. Otherwise, it is \textit{unwinnable}.

Linear equivalence, effectiveness, and winnability can be checked in \texttt{chipfiring} as follows:

\begin{lstlisting}[language=Python]
from chipfiring import linear_equivalence, is_winnable

divisor1.is_effective()
linear_equivalence(divisor1, divisor2)
is_winnable(divisor1)
\end{lstlisting}

\definition (cf. \cite[Definition 1.22]{corry2018divisors}) \label{complete-linear-system} A \textit{complete linear system} of $D \in \Div(G)$ is:
\[
|D| = \{E \in \Div(G) \colon E \sim D, E \geq 0\}.
\]
Equivalently, a \textit{complete linear system} is the set of all effective divisors on graph $G$ that are linearly equivalent to $D$.

\subsection{Laplacian \& Firing Script}

In general, a Laplacian aims to measure the ``equitability'' or evenness of a diffusive process in pure sciences. Similar to the continuous case, we define a discrete Laplacian by drawing some parallels. Before we start defining the Laplacian, let us generalize our notion of \textit{set-firing} of V from the previous section into a \textbf{firing-script}, where we plan to compactly encode the essential information for the move in which some vertices in lending subset V lend/borrow multiple times, for instance.

\definition (cf. \cite[Definition 2.2]{corry2018divisors}) \label{def_firingscript} A \textit{firing script} is a function $\sigma\colon V \to \mathbb{Z}$, which denotes the number of times each vertex $v$ lends (fires) if $\sigma(v) > 0$. If $\sigma(v) < 0$, it denotes the number of borrowing moves. Moreover, if $\sigma(v) = 0$, the vertex $v$ does not participate in the move. \footnote{\textbf{Aside:} The collection of all firing scripts form an abelian group $\mathcal{M}(G)$ or $\mathbb{Z}^V$ \cite[Definition 2.2]{corry2018divisors}.}

Furthermore, a \textit{discrete Laplacian operator} is used to map a firing script to a divisor, and we define it as follows:

\definition (cf.  \cite[Definition 2.1]{corry2018divisors}) \label{def_discreteLaplacian} The \textit{discrete Laplacian operator} on G is the linear mapping $L\colon\mathbb{Z}^V \to \mathbb{Z}^V$ defined by 
\[
L(f)(v) := \sum_{vw \in E} (f(v) - f(w)),
\] where the space \(\mathbb{Z}^V := \{ f\colon V \to \mathbb{Z} \}\) contains $\mathbb{Z}$-valued functions on the vertices of G.

In the context of chip firing games, we can think of the discrete Laplacian as a tool that maps a firing script in $M(G)$ to a resulting divisor in $\Div(G)$. If $\sigma\colon V \to \mathbb{Z}$ is a firing script, then the resulting divisor after firing is given by:
\[
D' = D - \sum_{v \in V} \sigma(v) \left( \val(v)v - \sum_{wv \in E} w \right) = D - \sum_{v \in V} \sigma(v) \sum_{vw \in E} (v - w)
\]
\[
= D - \sum_{v \in V} \left( \val(v)\sigma(v) - \sum_{vw \in E} \sigma(w) \right)v = D - \sum_{v \in V} \sum_{vw \in E} (\sigma(v) - \sigma(w))v
\]

Thus, any divisor can \& should be reached by a firing script from any other divisor in the linearly equivalent set.

\definition (cf. \cite[Definition 2.3, Exercise 2.4]{corry2018divisors}) \label{def_script-firing} The \textit{script-firing with firing script} $\sigma$ is denoted by $D \xrightarrow{\sigma} D'$, and because degree is preserved under firing moves, we also have $\deg(L(\sigma)) = 0$.

\definition (cf.  \cite[p. 20]{corry2018divisors}) The \textbf{principal divisor} associated to a firing script $\sigma$ is
\[
\text{div}(\sigma) := \sum_{v \in V} \left( \val(v) \cdot \sigma(v) - \sum_{vw \in E} \sigma(w) \right) v.
\] 

It is also worth noting that the set of principal divisors form a subgroup $\Prin(G) < \Div^0(G)$. This also means that linear equivalence class of D is coset of $\Prin(G)$: $[D] = D+ \Prin(G)$. From this, we also define the Picard Group as $\Pic(G) = \Div(G)/\Prin(G)$ and the Jacobian Group as $\Jac(G) = \Div^0(G)/\Prin(G)$. The Jacobian group is also called the \emph{critical group.}  \cite{baker2007riemann, biggs1997chip}

The matrix representation of the Laplacian operator will is the Laplacian matrix, which is can also be written as follows.

\definition (cf. \cite[Definition 2.6]{corry2018divisors}) The \textit{Laplacian matrix}, also denoted by $L$, is the $|V| \times |V|$ integer matrix with

\[
L_{ij} = L(\chi_j)(v_i) = 
\begin{cases} 
\val(v_i) & \text{if } i = j \\
-\text{(\# of edges between } v_j \text{ and } v_i\text{)} & \text{if } i \neq j.
\end{cases}
\]
Here \(\chi_j(v_i) = 
\begin{cases} 
1 & i = j \\
0 & i \neq j.
\end{cases}\) is the firing script of $v_j$ making a single lending move (to $v_i$).

It is also evident that $L = \text{Deg}(G) - A^T$, where $\text{Deg}(G)$ is a diagonal matrix with vertex-degrees of $G$, and $A$ is the adjacency matrix.
All lending moves are encoded in $L$ because the lending move by $v_j$ corresponds to subtracting the $j$th column of $L$ from a divisor at hand. For instance, given a firing-script column vector $\Vec{\sigma}$, we can say $D' = D - L\Vec{\sigma}$.

As an illustration, going back to the example we have discussed so far in Figure \ref{fig:sf-walkthrough}, we can see that the Laplacian matrix assumes the following form:

\(
L = 
\begin{bmatrix}
4 & -1 & -1 & -2 \\
-1 & 2 & -1 & 0 \\
-1 & -1 & 3 & -1 \\
-2 & 0 & -1 & 3 \\
\end{bmatrix} 
\), which is filled in with headers as
\(
\begin{array}{c|cccc}
& V_{Alice} & V_{Bob} & V_{Charlie} & V_{Elise} \\
\hline
V_{Alice} & 4 & -1 & -1 & -2 \\
V_{Bob} & -1 & 2 & -1 & 0 \\
V_{Charlie} & -1 & -1 & 3 & -1 \\
V_{Elise} & -2 & 0 & -1 & 3 \\
\end{array}
\)

Moreover, from Figure \ref{fig:sf-walkthrough}, we can see that to win (reach an effective divisor), Bob borrowed twice, and then Bob and Charlie both can be considered to lend once. So, we can represent this in the form of a firing script (ordered column vector) as: \(
 \Vec{\sigma} = 
\begin{bmatrix}
0 \\
-1 \\
1 \\
0 \\
\end{bmatrix} 
\). Thus, 

\(
D' = \begin{bmatrix}
2 \\
-3 \\
4 \\
-1 \\
\end{bmatrix} - 
\begin{bmatrix}
4 & -1 & -1 & -2 \\
-1 & 2 & -1 & 0 \\
-1 & -1 & 3 & -1 \\
-2 & 0 & -1 & 3 \\
\end{bmatrix}
\begin{bmatrix}
0 \\
-1 \\
1 \\
0 \\
\end{bmatrix} =
\begin{bmatrix}
2 \\
-3 \\
4 \\
-1 \\
\end{bmatrix} - 
\begin{bmatrix}
0 \\
-3 \\
4 \\
-1 \\
\end{bmatrix} =
\begin{bmatrix}
2 \\
0 \\
0 \\
0 \\
\end{bmatrix}
\)

We can see that the final result $D'$ matches the effective divisor we obtained in Figure \ref{fig:sf-walkthrough} right when we hit the win condition. The above example can be performed in \texttt{chipfiring} as follows:

\begin{lstlisting}[language=Python]
from chipfiring import CFGraph, CFLaplacian, CFiringScript
import numpy as np

# Create Laplacian for a CFGraph
laplacian = CFLaplacian(graph)

# Create a firing script (net firings at each vertex)
script = {"Charlie": 1, "Bob": -1}  # Charlie lends 1, Bob borrows 1
firing_script = CFiringScript(graph, script)

# Apply the Laplacian: D' = D - L * firing_script
result_divisor = laplacian.apply(initial_divisor, firing_script)
print(f"Result: {result_divisor.degrees_to_str()}")
\end{lstlisting}

\section{Algorithms for Winnability}
\label{ch:algorithms}
Having defined our setup and objective (winnability), we now consider some algorithms for winnability determination.

\subsection{Greedy Algorithm}

One way to play the dollar game is for each in-debt vertex to attempt to borrow its way out of debt. The problem is that borrowing once from each vertex is the same as not borrowing at all. Solving this gives an algorithm for the dollar game, which is essentially repeatedly choosing an in-debt vertex and making a borrowing move at that vertex until either the game is won or it becomes impossible to go on without reaching a state in which all of the vertices have made borrowing moves. Below, we first present an adaptation of the algorithm's pseudocode from \cite[\S 3.1]{corry2018divisors}.

\textbf{Note on Uniqueness of the greedy algorithm script:} The greedy algorithm \footnote{Proof of the validity of the greedy algorithm, and the uniqueness of generated firing script can be found in \cite[\S 3.1]{corry2018divisors}.} can be modified to produce a firing script if its input is winnable. Initialize by setting \( \sigma = 0 \), and then each time step 6 of the algorithm below is invoked, replace \( \sigma \) by \( \sigma - v \). It turns out that the resulting script is independent of the order in which vertices are added.

\begin{algorithm}
\label{greedy-algo}
\caption{Greedy algorithm for the dollar game. (Adapted from \cite[Algorithm 1]{corry2018divisors})}
\begin{algorithmic}[1]
\REQUIRE $D \in \Div(G)$.
\ENSURE TRUE if $D$ is winnable; FALSE if not.
\STATE \textbf{initialization:} $M = \emptyset \subseteq V$, the set of marked vertices.
\WHILE{$D$ not effective}
    \IF{$M \neq V$}
        \STATE choose any vertex in debt: $v \in V$ such that $D(v) < 0$
        \STATE modify $D$ by performing a borrowing move at $v$
        \IF{$v$ is not in $M$}
            \STATE add $v$ to $M$
        \ENDIF
    \ELSE 
        \STATE /* required to borrow from all vertices */
        \RETURN FALSE \hfill /* unwinnable */
    \ENDIF
\ENDWHILE
\RETURN TRUE \hfill /* winnable */
\end{algorithmic}
\end{algorithm}

Let's now introduce the greedy algorithm Pythonic API from our \texttt{chipfiring} package.

\begin{lstlisting}[language=Python]
from chipfiring import CFGraph, CFDivisor, GreedyAlgorithm

# Initialize Greedy Algorithm with a CFGraph and a CFDivisor
algorithm = GreedyAlgorithm(graph, divisor)

# Run the algorithm to determine winnability
# The winnability determination and a CFiringScript object are returned
winnable, firing_script = algorithm.play()
\end{lstlisting}

\subsection{``Benevolence'' Algorithm}
\label{sec:benevolence-algo}

Rather than implementing greed, one might try its opposite: a ``benevolent'' vertex $q$ might choose to concentrate all the games' debt upon itself. One particular form of benevolence that is of significant theoretical and algorithmic importance is encapsulated in the following notion. Informally, one can view this notion as follows: the vertex $q$ volunteers to be the only vertex in debt, on the condition that the remaining vertices move as much money as possible in $q$'s direction, without going into debt.

\definition (cf. \cite[Definition 3.4]{corry2018divisors}) \label{def_qreduced} Let $q \in V$. A divisor $D \in \Div(G)$ is called \textbf{$q$-reduced} if both
    \begin{enumerate}
        \item $D(v) \geq 0$ for all $v \in V \setminus \{q\}$, and
        \item For every nonempty subset $S \subseteq V \setminus \{q\}$, there exists a vertex $v \in S$ such that $D(v) < \outdeg_S(v)$, where $\outdeg_S(v)$ denotes the number of edges $vw$ such that $w \notin S$.
    \end{enumerate}
    
   It is convenient to have a shorthand for this second condition, which is provided by the next definition.
   
   \definition (cf. \cite[Definition 3.3]{corry2018divisors}) \label{legal-set-firing} Let \( D \in \Div(G) \), and let \( S \subseteq V \). Suppose \( D' \) is obtained from \( D \) by firing each of the vertices in \( S \) once. Then \( D \xrightarrow{S} D' \) is a \textbf{legal set-firing} if \( D'(v) \geq 0 \) for all \( v \in S \), i.e., after firing \( S \), none of the vertices in \( S \) are in debt. In this case, we say it is \textbf{legal to fire \( S \)}. [Note: if it is legal to fire \( S \), then the vertices in \( S \) must also be out of debt \textit{before} firing.]
   
  Now observe that the second condition of Definition \ref{def_qreduced} may be rephrased as: it is not legal to fire any nonempty set $S$ that does not contain $q$. It is in this sense that the chips of $D$ cannot be pulled any closer to $q$. The existence of such a divisor is neatly established by a minimization argument, using the following order.

\definition (cf. \cite[Exercise 3.7]{corry2018divisors}) \label{def_prec} 
Given divisors $D, D' \in \operatorname{Div}(G)$ and a spanning tree $(T, q)$ of $G$ rooted at a vertex $q$, let $v_1 = q, v_2, \dotsc, v_n$ be a tree ordering of the vertices, where:
\begin{itemize}
    \item $T$ is a connected, cycle-free subgraph of $G$ that includes all vertices and contains exactly $n-1$ edges (i.e., a spanning tree),
    \item the ordering respects the structure of $T$, meaning that if $v_i$ lies on the unique path from $q$ to $v_j$ in $T$, then $i < j$.
\end{itemize}
We say that $D' \prec D$ if either:
\begin{enumerate}
    \item $\deg(D') < \deg(D)$, or
    \item $\deg(D') = \deg(D)$ and there exists an index $i$ such that $D'(v_i) > D(v_i)$, and for all $j < i$, $D'(v_j) = D(v_j)$. Equivalently, $D \neq D'$ and, for the minimum index $i$ such that $D'(v_i) \neq D(v_i)$, we have $D'(v_i) > D(v_i)$.
\end{enumerate}

In other words, $\prec$ is the reverse lexicographic order on divisors of the same degree, where the coefficients are written according to the chosen tree ordering. The utility of $\prec$ is that it encapsulates the idea that $D' \prec D$ if the chips of $D'$ have been ``pulled closer to'' $q$, in a sense.

\begin{proposition} (cf. \cite[Exercise 3.7]{corry2018divisors})
\label{prop:qreduced-minimum}
Fix a divisor $D$, and let $\mathcal{D}$ denote the set of all divisors $D' \sim D$ such that $D'(v) \geq 0$ for all $v \neq q$. Suppose that $D'$ is minimal in $\mathcal{D}$ under the order $\prec$. Then $D'$ is $q$-reduced.
\end{proposition}

\begin{proof}
Suppose to the contrary that $D' \in \mathcal{D}$ is not $q$-reduced. Then there is some legal set-firing $D' \overset{S}{\rightarrow} D''$ where $\emptyset \neq S \subseteq V \backslash{q}$. Let $v_i \in S$ be the minimum element of $S$ under the chosen tree ordering, and let $v_j$ be the parent of $v_i$ in the spanning tree. Then firing $S$ sends at least one chip from $v_i$ to $v_j$. Furthermore, none of the vertices $v_2, \cdots, v_j$ is in $S$, so none of these vertices lose chips when firing $S$. Since at least one of them gains a chip, this implies that $D'' \prec D'$, so $D'$ is \emph{not} minimal.
\end{proof}

In fact, the converse of Proposition \ref{prop:qreduced-minimum} is also true, though we omit the proof; see \cite[Theorem 3.6]{corry2018divisors}. That is, the $q$-reduced divisors $D_q \sim D$ is \emph{unique}. It follows from definitions that if \emph{any} $D' \sim D$ is effective, then $D_q$ must be. Therefore:

\begin{proposition} (cf. \cite[Corollary 3.7]{corry2018divisors})
\label{prop:lin-eq-q-red}
Let \( D \in \Div(G) \), and let \( D' \) be the \( q \)-reduced divisor linearly equivalent to \( D \). Then \( |D| \neq \emptyset \) if and only if \( D' \geq 0 \). In other words, \( D \) is winnable if and only if \( D'(q) \geq 0 \).
\end{proposition}


From this, we obtain a version of ``benevolence'' guaranteed to produce an effective $E \sim D$ if it exists. This can be done by following the steps, which are adapted from \cite[\S 3.2]{corry2018divisors}.

\begin{enumerate}
    \item Pick some ``benevolent vertex'' \( q \in V \). Call \( q \) the source vertex, and let \( V \setminus \{q\} \) be the set of non-source vertices.
    \item Let \( q \) lend/fire so many chips that the non-source vertices, sharing among themselves, are out of debt. This is done to concentrate the debt at the source vertex.
    \item At this stage, only \( q \) is in debt, and it makes no further lending or borrowing moves. It is now the job of the non-source vertices to try to relieve \( q \) of its debt. Look for \( S \subseteq V \setminus \{q\} \) with the legal set-firing property as stated in definition \ref{legal-set-firing}. Having found such an \( S \), make the corresponding set-lending move. 
    \item Repeat until no such \( S \) remains. The resulting divisor is said to be \( q \)-reduced. More importantly, if, in the end, $q$ is no longer in debt, $D$ is winnable. Otherwise, $|D| = \emptyset$, or equivalently $D$ is unwinnable.
\end{enumerate}

In the \texttt{chipfiring} package, this procedure can be performed with the following syntax.

\begin{lstlisting}[language=Python]
from chipfiring import q_reduction, is_q_reduced

# Given any divisor, return a q-reduced divisor
q_red_div = q_reduction(divisor)

# Check if any given divisor is q_reduced
is_q_reduced(divisor)
\end{lstlisting}

As mentioned in \cite{corry2018divisors}, some naturally interesting questions about this strategy are: Is it always possible to complete step 2? Is step 3 guaranteed to terminate? If the strategy does not win, does this mean the game is unwinnable? (After all, the moves in step 3 are constrained.) Is the resulting \( q \)-reduced divisor unique? Can the strategy be efficiently implemented?

The following sections gradually show that the answer to all of these questions is ``yes,'' and explain the implementation of $\texttt{q\_reduction}.$

\subsection{Configurations \& related preliminaries}

In many computations, it is useful to choose a special vertex $q$ and ignore any chips at $q$. Following \cite{corry2018divisors}, we use the word \emph{configuration} for a divisor without a specified degree at $q$.

\definition (cf. \cite[\S 2.2]{corry2018divisors}) \label{def:configuration} Fix a vertex $q \in V$ and define $\widetilde{V} := V \setminus \{q\}$. Then a \textit{configuration} $c$ is an element of the subgroup
\[
\text{Config}(G, q) = \mathbb{Z}\widetilde{V} \subseteq \mathbb{Z}V = \Div(G).
\]

We also write $c' \geq c$ for $c,c' \in \text{Config}(G)$ if $c(v) \geq c'(v)$ for all $v \in \widetilde{V}$. A configuration $c$ is said to be \textit{non-negative} ($c \geq 0$) is $c(v) \geq 0$ for all $v \in \widetilde{V}$. We define lending \& borrowing operations on configurations the same way as with divisors, but in the case of configurations, we do not keep track of the number of chips present at $q$.

\definition (cf. \cite[\S 2.2]{corry2018divisors}) The \textit{degree} of a configuration \( c \) is calculated as $\deg(c) = \sum_{v \in \widetilde{V}} c(v)$.

\definition (cf. \cite[\S 2.2]{corry2018divisors}) Configurations \( c \) and \( c' \) are said to be \textit{linearly equivalent}, written as \( c \sim c' \) if they can be transformed into one another through a sequence of lending and borrowing operations. 

\textbf{Note:} Unlike linearly equivalent divisors, linearly equivalent configurations need not have the same degree. In effect, the difference of degrees tells how many chips were fired to or from $q$. More precisely, $c \sim c'$ as configurations if and only if $c - \deg(c) q \sim c' - \deg(c') q$ as divisors.

For instance, let us consider a slight relabeling of the example in Figure \ref{fig:div-setup}, where Bob is labeled as $q$, and thus we consider configurations with respect to Bob. We can see configurations $c \sim c'$ as depicted in Figure \ref{fig:config-equivalence}.

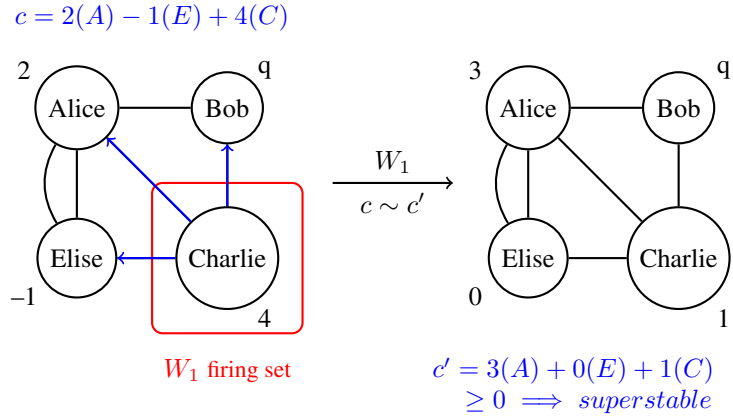
\begin{figure}[h]
    \centering
    \begin{tikzpicture}[scale=1, transform shape, thick]
      \node[circle,draw] (A1) at (0,4) {Alice};
      \node[circle,draw] (B1) at (2,4) {Bob};
      \node[circle,draw] (E1) at (0,2) {Elise};
      \node[circle,draw] (C1) at (2,2) {Charlie};
      \draw (A1)--(B1)--(C1)--(E1)--(A1);
      \draw (A1)--(C1);
      \draw (E1) to[bend left] (A1);
      \node at (-0.7,4.5) {2};
      \node at (2.5,4.5) {q};
      \node at (-0.7,1.5) {--1};
      \node at (2.5,1.2) {4};
      \draw[red,thick,rounded corners] (1,3)--(3,3)--(3,1)--(1,1)--cycle;
      \node[red] at (2,0.5) {$W_1$ \small firing set};

      \node[blue] at (1,5.2) {$c=2(A)-1(E)+4(C)$};

      \draw[->,blue] (C1)--(A1);
      \draw[->,blue] (C1)--(B1);
      \draw[->,blue] (C1)--(E1);

      \node[circle,draw] (A2) at (6,4) {Alice};
      \node[circle,draw] (B2) at (8,4) {Bob};
      \node[circle,draw] (E2) at (6,2) {Elise};
      \node[circle,draw] (C2) at (8,2) {Charlie};
      \draw (A2)--(B2)--(C2)--(E2)--(A2);
      \draw (A2)--(C2);
      \draw (E2) to[bend left] (A2);
      \node at (5.3,4.5) {3};
      \node at (8.6,4.5) {q};
      \node at (5.3,1.5) {0};
      \node at (8.6,1.2) {1};

      \node[blue] at (6.6,0.5) {$c'=3(A)+0(E)+1(C)$};
      \node[blue] at (6.8,0.1) {$\geq 0\implies superstable$} ;
      
      \draw[->,thick] (3.4,3)--(5,3) node[midway,above] {$W_1$} node[midway, below] {$c \sim c'$};

    \end{tikzpicture}
    \caption{Application of firing move by Charlie, on configurations with respect to $q = \mathrm{Bob}$. This is the same graph as Figure \ref{fig:div-setup}.}
    \label{fig:config-equivalence}
\end{figure}

The notion of a $q$-reduced divisor corresponds to the notion of \emph{superstable} configuration, as defined below.

\definition (cf. \cite[Definition 3.11]{corry2018divisors}) Let \( c \in \text{Config}(G) \), and let \( S \subseteq \widetilde{V} \). Suppose \( c' \) is the configuration obtained from \( c \) by firing the vertices in \( S \). Then \( c \xrightarrow{S} c' \) is a \textbf{legal set-firing} if \( c'(v) \geq 0 \) for all \( v \in S \).

\definition (cf. \cite[Definition 3.12]{corry2018divisors}) \label{def:superstable} The configuration \( c \in \text{Config}(G) \) is \textbf{superstable} if \( c \geq 0 \) and has no legal nonempty set-firings. Equivalently, for all nonempty \( S \subseteq \widetilde{V} \), there exists \( v \in S \) such that \( c(v) < \text{outdeg}_S(v) \).

Now, let's introduce our \texttt{chipfiring} package API corresponding to these notions.

\begin{lstlisting}[language=Python]
from chipfiring import CFConfig

# Given any divisor, and a vertex q, initialize a configuration
conf_wrt_Bob = CFConfig(divisor, q_name="Bob")

# Retrieve degrees at non-q vertices of the configuration
degree_of_conf_wrt_Bob = conf_wrt_Bob.get_degree_sum()

# Calculate outdegree of the configuration
# with respect to vertex v and set S
conf_wrt_Bob.get_out_degree_S(
    v_name_in_S = "Alice",
    S_names = {"Alice", "Charlie"}
)

# Check if a given set S is legal set-firing
conf_wrt_Bob.is_legal_set_firing(
    S_names = {"Alice", "Charlie"}
)

# Check superstability of configuration
conf_wrt_Bob.is_superstable()

# Retrieving underlying properties of configuration
print(f"Degree at q: {config.get_q_underlying_degree()}")
print(f"Configuration on V~: {config.get_v_tilde_names()}")
\end{lstlisting}

\subsection{Dhar's Algorithm}
\label{ch:dhar-algo} 

Upon rebranding $q$-reducedness as superstability, we can rephrase the task set before us: given a configuration $c$, we wish to repeatedly find legal sets to fire until we arrive at a superstable configuration. Naturally, we wish to find such sets without having to search through the entire parameter space of $2^{|\widetilde{V}|} - 1$ plausible subsets. 
A beautiful answer to this need is provided by \emph{Dhar's burning algorithm}. This algorithm was introduced in \cite{dhar1990self}; we follow the exposition of it in \cite[\S 3.4.1]{corry2018divisors}.

Let $c \in \text{Config}(G, q)$, and assume $c \geq 0$  (if we begin with a non-effective configuration, we should first lend generously from $q$ and spread the wealth). Now, to find a legal set-firing for $c$, if it exists, imagine the edges of our graph are made of wood so that when vertex \( q \) is ignited, the fire spreads along its incident edges. Furthermore, think of the configuration \( c \) as \( c(v) \) firefighters present at each \( v \in \widetilde{V} \), given that each firefighter can only control the fire coming from a single edge. This tells us that a vertex is protected only if the number of burning incident edges is $\leq c(v)$. Otherwise, the firefighters fail, and the vertex is set on fire.\footnote{Fun Note (cf. \cite[p. 46]{corry2018divisors}): No need to worry; firefighters are rescued by an underground tunnel built by Amherst College.} In the end, we conclude that the unburnt vertices constitute a set that may be legally fired from \( c \) and that if this set is empty, then by Definition \ref{def:superstable}, \( c \) is superstable. We adapt below the pseudocode version of the algorithm from \cite[\S 3.4.1]{corry2018divisors}.

\begin{algorithm}
\caption{Dhar's algorithm.}
\begin{algorithmic}[1]
\REQUIRE a nonnegative configuration \( c \)
\ENSURE a legal firing set \( S \subseteq \widetilde{V} \), empty iff \( c \) is superstable
\STATE \textbf{initialization:} \( S = \widetilde{V} \)
\WHILE{ \( S \neq \emptyset \) }
    \IF{ \( c(v) < \text{outdeg}_S(v) \) for some \( v \in S \)}
        \STATE \( S = S \setminus \{v\} \)
    \ELSE
        \RETURN \( S \) \hfill /* \( c \) is not superstable */
    \ENDIF
\ENDWHILE
\RETURN \( S \)
\end{algorithmic}
\end{algorithm}

Let us run through Dhar's algorithm described above in Figure \ref{fig:dhar-example} while continuing from the starting linearly equivalent configuration ($c'$), which we obtained in Figure \ref{fig:config-equivalence}. After the vertex $q$ is set on fire, it sets two edges on fire (one between Alice and Bob and one between Bob and Charlie). However, since there are $3 \geq 1$ firemen at Alice's vertex, and there are $1 \geq 1$ firefighters at Charlie's vertex, the algorithm terminates and outputs the legal set firing as the set $S = \{\mathrm{Alice},  \mathrm{Elise}, \mathrm{Charlie}\}$ as expected.

\begin{figure}[h]
    \centering
    \begin{tikzpicture}[scale=1, transform shape, thick]
      \node[circle,draw] (A1) at (0,4) {Alice};
      \node[circle,draw,red] (B1) at (2,4) {Bob};
      \node[circle,draw] (E1) at (0,2) {Elise};
      \node[circle,draw] (C1) at (2,2) {Charlie};
      \draw (A1)--(B1)--(C1)--(E1)--(A1);
      \draw (A1)--(C1);
      \draw (E1) to[bend left] (A1);
      \node at (-1,4.5) {3};
      \node[blue] at (-1,4) {($>$1)};
      \node[blue] at (-1.4,3.6) {Protected!};
      \node at (2.5,4.5) {q};
      \node at (-0.7,1.5) {0};
      \node at (2.35,1.1) {1};
      \node[blue] at (2.9,1.1) {($\geq$ 1)};
      \node[blue] at (2.4,0.7) {Protected!};
      \draw[red,thick,rounded corners] (-2.3,4.8)--(1.15,4.8)--(1.15,3.1)--(3.4,3.1)--(3.4,0.3)--(-2.3,0.3)--cycle;
      \node[red] at (0,5.2) {Legal firing set S};

      \node[blue] at (6,3.2) {Algorithm Terminates here \&};
      \node[blue] at (6,2.8) {outputs the Legal firing set S};
      \node[blue] at (6,1.2) {$c=3(A)+0(E)+1(C) \geq 0$};

      \draw[-,red,thick] (A1)--(B1);
      \draw[-,red] (C1)--(B1);

    \end{tikzpicture}
    \caption{Application of Dhar's Algorithm on the same graph as Figure \ref{fig:config-equivalence}}
    \label{fig:dhar-example}
\end{figure}
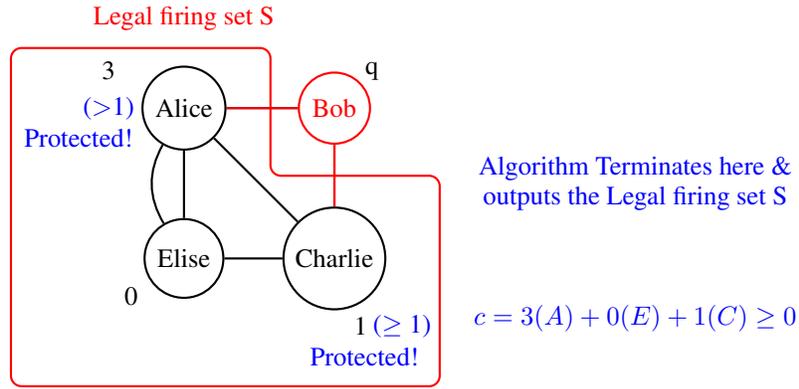

Since we have Dhar's algorithm as a tool, let us revisit the procedure of finding $q$-reduced divisors (the ``benevolence'' algorithm) as described in Section \ref{sec:benevolence-algo}. Below, we present an adaptation of the pseudocode version of the updated algorithm described in \cite[\S3.4.2]{corry2018divisors}. This algorithm has some additional peculiar optimizations, which can help decrease the runtime.

\begin{algorithm}
\label{lin-eq-q-red}
\caption{Find the linearly equivalent \( q \)-reduced divisor.}
\begin{algorithmic}[1]
\REQUIRE \( D \in \Div(G) \) and \( q \in V \)
\ENSURE the unique \( q \)-reduced divisor linearly equivalent to \( D \)
\STATE Use a greedy algorithm to bring each vertex \( v \neq q \) out of debt, so that we may assume \( D(v) \geq 0 \) for all \( v \neq q \).
\STATE Repeatedly apply Dhar's algorithm until \( D \) is \( q \)-reduced.
\end{algorithmic}
\end{algorithm}

The Dhar's algorithm implementation from \texttt{chipfiring} can be used as follows; it returns the firing set as well as an "orientation," which is defined in Definition \ref{def_orientation}:

\begin{lstlisting}[language=Python]
from chipfiring import CFGraph, CFDivisor, DharAlgorithm

# Initialize Dhar's algorithm with a CFGraph, CFDivisor, and the name of the `q` vertex
dhar_algo = DharAlgorithm(graph, divisor, "Bob")

# Run the algorithm to find the maximal legal firing set
firing_set, orientation = dhar_algo.run()
\end{lstlisting}

\subsection{An Efficient Winnability Determination Algorithm}

Now, using all the tools \& algorithms we have developed so far, let's devise an efficient winnability determination algorithm. We formalize the algorithm in the form of pseudocode below. 
Recall that $D_q(q) \geq 0$ directly translates to \textbf{winnability} by Proposition \ref{prop:lin-eq-q-red}.

\begin{algorithm}
\caption{Efficient Winnability Determination Algorithm}
\label{winnability-algo}
\begin{algorithmic}[1]
\REQUIRE $D \in \Div(G)$
\ENSURE TRUE if $D$ is winnable; FALSE if not
\STATE Choose source vertex $q \in V$
\STATE Let $\widetilde{V} \gets V \setminus \{q\}$ \hfill /* Set of non-source vertices */
\STATE Fire from $q$ and share among vertices in $\widetilde{V}$ until only $q$ is in debt
\WHILE{Dhar's Algorithm returns a non-empty set}
    \STATE Apply Dhar's Algorithm to current configuration $c$
    \STATE Fire the returned set if non-empty
\ENDWHILE
\STATE $D_q(q) \gets \deg(D) - \deg(c)$ \hfill /* Calculate chips on source vertex */
\IF{$D_q(q) \geq 0$}
    \RETURN TRUE  \hfill /* winnable */
\ELSE
    \RETURN FALSE  \hfill /* unwinnable */
\ENDIF
\end{algorithmic}
\end{algorithm}

Furthermore, one of the strategies that can be used in step 3 of the algorithm below is to systematically concentrate all debt at our distinguished vertex $q$ through a reverse-distance prioritized approach. The key insight is that we can move debt away from vertices furthest from $q$ first, working our way inward toward $q$ systematically.

This strategy works by ordering all non-source vertices based on their distance from q, which we can determine through a simple \textit{breadth-first search (BFS) traversal}. Then, proceeding from the vertices furthest from $q$ and moving inward, we perform borrowing operations on any vertex with a negative chip count (in debt). When a vertex borrows, it receives chips equal to its degree but must distribute one chip along each incident edge to its neighbors. This borrowing operation effectively pushes debt closer to $q$.

By processing vertices in reverse distance order (from furthest to closest to q), we ensure that once a vertex is out of debt, it remains out of debt throughout the process. This is because we only process vertices closer to q after all vertices further from q have been handled. The result is a configuration where only the source vertex q may be in debt, with all other vertices having non-negative chip counts. In practice, the number of Dhar iterations is typically small. This makes the algorithm more efficient than exhaustive approaches that might simulate all possible chip-firing sequences. 

As an illustration, we implemented this BFS-based debt clustering strategy along with Dhar's algorithm as a Python API in the \texttt{chipfiring} package. Upon running our implementation on the example from Figure \ref{fig:config-equivalence}, we can see that it outputs the following: ``The game is winnable using Dhar's algorithm with a legal firing set of {Alice, Charlie, Elise}''.

An example demonstrating how to use our package's implementation of the Efficient Winnability Determination Algorithm is below. Note that our implementation has an optional \texttt{optimized} mode that utilizes various matrix optimizations, theorems, lemmas, and properties to accelerate determination; note that you might not get the associated induced orientation and q-reduced divisor under this mode. There is also an optional \texttt{visualize} parameter that, if set to true, automatically generates an interactive visualization that will allow you to view the steps taken by \texttt{EWD} to generate the determination. The visualization of running EWD on the example from \autoref{fig:div-setup} can be seen at \autoref{fig:EWD-viz}.

\begin{lstlisting}[language=Python]
from chipfiring import CFGraph, CFDivisor, EWD

# Given a CFGraph and CFDivisor, run the EWD algorithm to check for winnability.
# The function returns the winnability status, the q-reduced divisor, 
# the final edge orientation, and a visualizer object (if enabled).
is_winnable, q_reduced, orientation, visualizer = EWD(
    graph, 
    divisor, 
    optimized = True, 
    visualize = True
)

# Visualize
# Note that this launches a web-based interactive visualization at localhost:8050,
# so the program will halt until the user manually closes the visualization
# (Ctrl + C in the terminal)
visualizer.visualize()
\end{lstlisting}

\begin{figure}
    \centering
    \includegraphics[width=0.5\linewidth]{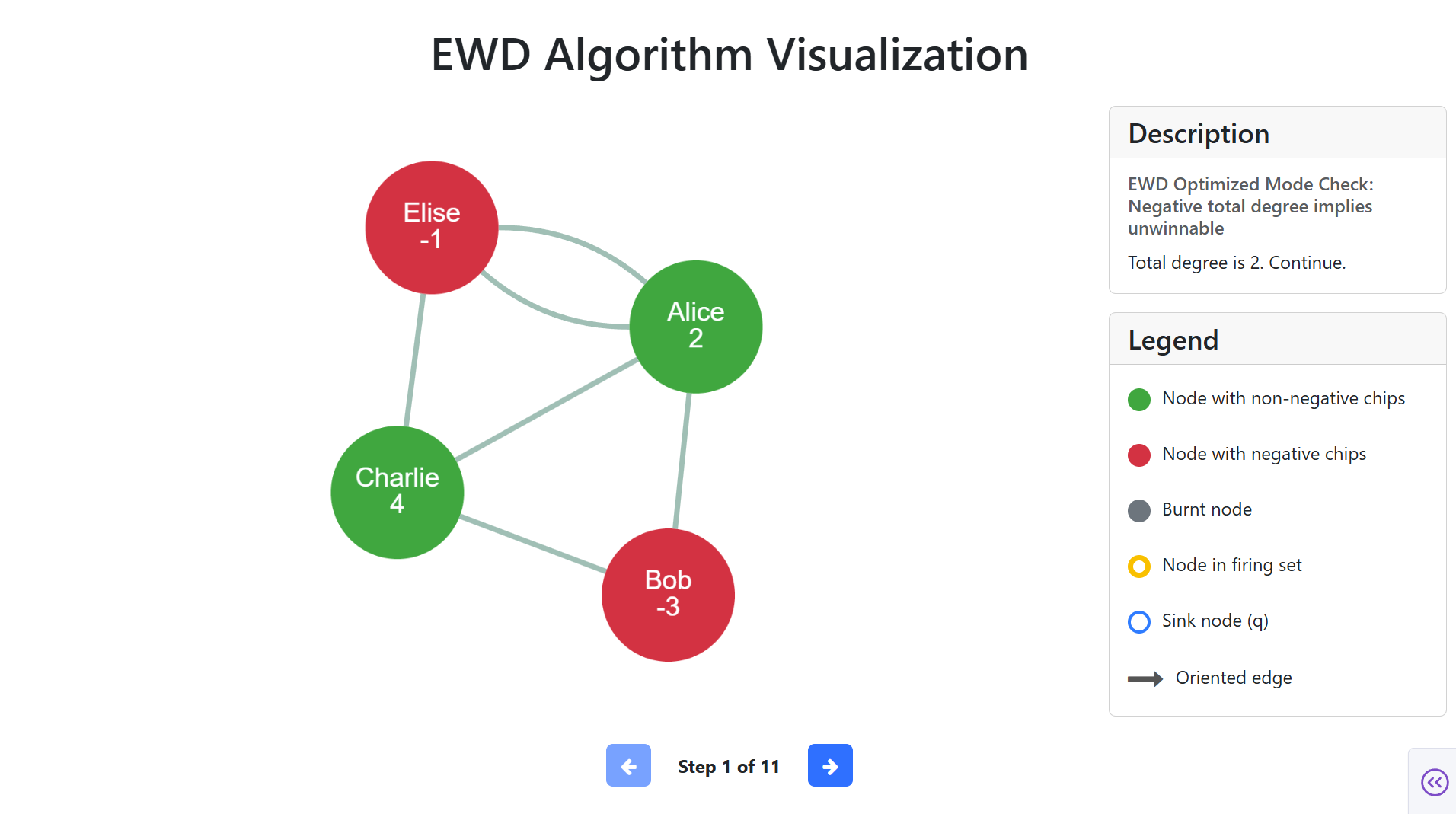}
    \caption{Screenshot of an example visualization of the EWD Algorithm}
    \label{fig:EWD-viz}
\end{figure}

\section{Riemann-Roch for Graphs \& Rank Determination}
\label{ch:Riemann-roch}

In this section, we will build on some final tools, mainly including orientations and Baker--Norine ranks, which will, in turn, help us state and better understand the Riemann--Roch theorem for graphs. This will help us to quantify ``winnability,'' and study algorithms to quickly do the same. The notion of ``rank'' considered herein originated in \cite{baker2007riemann}, where the graph Riemann--Roch theorem was first stated and proved. The original proof hinges crucially on a set, denoted by $\mathcal{N}$ in \cite{baker2007riemann} and sometimes called ``moderators,'' with two important interpretations: on the one hand, as maximal unwinnable divisors, and on the other hand as divisors arising from graph orientations. Therefore our \texttt{chipfiring} package provides tools for both the Baker--Norine rank and graph orientations.

\definition \textit{Maximal unwinnable divisors} are those unwinnable divisors $D$ such that for any unwinnable divisor $D'$, if $D \leq D'$, then $D = D'$. Equivalently, $D$ is maximal unwinnable if $D$ is unwinnable, but $D + v$ is winnable for each $v \in V$.

\definition \label{def:max-superstable} \textit{Maximal superstable configurations} are those superstable configurations $c$ such that for any superstable configuration $c'$, if $c \leq c'$, then $c = c'$.

\subsection{Orientations and Genus}

\definition \label{def_orientation} A graph orientation $\mathcal{O}$ assigns a specific direction to each edge of the graph within the multiset of edges. For any edge $e = uv \in E$, we define $e^- = u$ as the starting vertex (tail) and $e^+ = v$ as the ending vertex (head), meaning that the edge $e$ is directed from $u$ to $v$.

\definition The \textit{reverse orientation} of $\mathcal{O}$, denoted by $\overline{\mathcal{O}}$, swaps the roles of the head and the tail of each edge. For instance, $e$ under $\overline{\mathcal{O}}$ would become $\overline{e}$ with $\overline{e}^- = v$ as the starting vertex (tail) and $\overline{e}^+ = u$ as the ending vertex (head).

A vertex $w$ is called a \textit{source} for an orientation $\mathcal{O}$ if all edges incident to $w$ are directed away from $w$. Equivalently, if for all edges $e \in \mathcal{O}$, $e^+ \neq w$. Similarly, a vertex $v$ is called a \textit{sink} for an orientation $\mathcal{O}$ if all edges incident to $v$ are directed towards $v$. Equivalently, for all edges $e \in \mathcal{O}$, $e^- \neq v$.

\definition For any vertex $v \in V$ under an orientation $\mathcal{O}$, the \textit{outdegree} counts the edges that start from $v$ (i.e. $\operatorname{outdeg}_{\mathcal{O}}(v) = |\{e \in \mathcal{O} \colon e^- = v\}|$), while the \textit{indegree} is the number of edges directed towards $v$ (i.e. $\operatorname{indeg}_{\mathcal{O}}(v) = |\{e \in \mathcal{O} \colon e^+ = v\}|$).

\definition (cf. \cite[Definition 4.8]{corry2018divisors})  A \textit{divisor} associated with a given orientation $\mathcal{O}$ on the graph $G$ is defined as:

\[
D(\mathcal{O}) = \sum_{v \in V} (\operatorname{indeg}_{\mathcal{O}}(v) - 1) \cdot v.
\]

\definition(cf. \cite[Definition 4.8]{corry2018divisors})  \label{def:conf-of-orientation} The \textit{configuration} associated to a source vertex $q \in V$ under $\mathcal{O}$ is defined as:

\[
c(\mathcal{O}) = \sum_{v \in \widetilde{V}} (\operatorname{indeg}_{\mathcal{O}}(v) - 1) \cdot v.
\]

\definition \label{def:canonical} (cf. \cite[Definition 5.7]{corry2018divisors}) The \textit{canonical divisor} $K$ of a graph $G$ is defined as: \( K := D(\mathcal{O}) + D(\overline{\mathcal{O}}) \). The canonical divisor only depends on graph $G$ and is independent of orientation because for any vertex $v \in V$, we have
\[ K(v) = (\operatorname{indeg}_{\mathcal{O}}(v) - 1) + (\operatorname{outdeg}_{\mathcal{O}}(v) - 1) = \val(v) - 2.\]  and thus, the canonical divisor can also be written as: \( K := \sum_{v \in V} (\val(v) - 2) \cdot v. \)

The name of the ``canonical divisor'' arises from algebraic geometry. See for example \cite[Remark 4.19]{baker2008}. 

\definition A \textit{directed path} is a sequence of vertices connected by edges, where each vertex (except the first and last) acts as both the head of the previous edge and the tail of the next one. 

\textbf{Note:} All vertices along a directed path are distinct, except possibly the start \& end vertices. 

\definition A \textit{directed cycle} is a directed path in which the start and end vertices are identical.

\definition An orientation $\mathcal{O}$ is \textit{acyclic} if it does not contain any cycle of directed edges. 

\textbf{Note:} In the case of acyclic orientations, multiple edges between two vertices must be oriented in the same direction.

\begin{lemma}
\label{lem:acyclic-indeg-determination}
    An acyclic orientation can be determined by its indegree sequence: if $\mathcal{O}$ and $\mathcal{O'}$ are acyclic orientations of $G$ and $\operatorname{indeg}_{\mathcal{O}}(v) = \operatorname{indeg}_{\mathcal{O'}}(v)$ for all $v \in V$, then $\mathcal{O}=\mathcal{O'}$.
\end{lemma}

See \cite[Lemma 4.3]{corry2018divisors} for a proof.

If we revisit Dhar’s algorithm, as introduced in Section \ref{ch:dhar-algo}, we see that a bit more content may be extracted from it: the path of the fire tells us an acyclic orientation. This acyclic orientation is useful in other problems. We include the pseudocode for modified Dhar's algorithm below. This algorithm is based on the discussion in \cite[\S 4.2]{corry2018divisors}.

\begin{algorithm}
\caption{Orientation-based Dhar's Algorithm}
\begin{algorithmic}[1]
\REQUIRE a nonnegative configuration $c$ and source vertex $q$
\ENSURE a pair $(S, \mathcal{O})$ where $S \subseteq \widetilde{V}$ is a legal firing set (empty if and only if $c$ is superstable) and $\mathcal{O}$ is the resulting orientation
\STATE \textbf{initialization:} $S \gets \widetilde{V}$, $\mathcal{O} \gets \emptyset$, $B \gets \{q\}$ /* $B$ is burning set */
\WHILE{$B \neq V$}
    \FOR{each $v \in S$}
        \STATE $E_v \gets$ edges between $v$ and $B$ /* potentially burning edges */
        \IF{$|E_v| > c(v)$}
            \STATE $S \gets S \setminus \{v\}$
            \STATE $B \gets B \cup \{v\}$
            \FOR{each $e \in E_v$}
                \STATE Orient $e$ towards $v$ in $\mathcal{O}$ /* record burning direction */
            \ENDFOR
        \ENDIF
    \ENDFOR
    \IF{no new vertices added to $B$}
        \RETURN $(S, \mathcal{O})$ /* $c$ is not superstable */
    \ENDIF
\ENDWHILE
\RETURN $(\emptyset, \mathcal{O})$  /* $c$ is superstable */
\end{algorithmic}
\end{algorithm}

\newpage

In the package, Dhar's algorithm returns the orientation in the form of a \texttt{CFOrientation} object:

\begin{lstlisting}[language=Python]
from chipfiring import CFGraph, CFDivisor, DharAlgorithm

# Run Dhar's algorithm
dhar = DharAlgorithm(graph, divisor, "Alice")
unburnt_vertices, orientation = dhar.run()

# Examine the resulting orientation
# Note that the dict will omit any unoriented vertices
orientation_dict = orientation.to_dict()[orientations]
print(orientation_dict)
\end{lstlisting}

A \texttt{CFOrientation} object can also be initialized directly for experimentation as follows: 

\begin{lstlisting}[language=Python]
# Edges between Alice and Bob are directed from Alice to Bob
# Edges between Bob and Charlie are directed from Bob to Charlie
# All other edges are unoriented
orientations = [("Alice", "Bob"), ("Bob", "Charlie")]
orientation = CFOrientation(graph, orientations)
\end{lstlisting}

The crucial facts linking orientations to winnability is the following. Proofs may be found in \cite{corry2018divisors}. See also \cite{baker2007riemann}.

\begin{proposition}(\cite[Theorem 4.8]{corry2018divisors})
\label{prop:stable-bij}
Fix $q \in V$. Then, the correspondence \(\mathcal{O} \mapsto c(\mathcal{O}).\) is a bijection between acyclic orientations $\mathcal{O}$ of $G$, with unique source q, and maximal superstable configurations $c(\mathcal{O}) \in \operatorname{Config}(G, q)$.
\end{proposition} 

\definition \label{def_genus} The \textit{genus} of a graph is its topological Euler characteristic: \( g = |E| - |V| + 1. \)

The word \emph{genus} is sometimes used for other notions in graph theory. Here, we define it in this way so that it mirrors the genus in the classical Riemann--Roch formula for algebraic curves. Alternatively, in the setup of \cite{baker2008} (for example), genus $g$ algebraic curves may be specialized / tropicalized to genus $g$ graphs.

\begin{proposition}  (\cite[Corollary 4.9(1,2)]{corry2018divisors})
\label{prop:conf-div-max-unwinnable}
\label{max-unwinnable}
Let $c$ be a superstable configuration and $D$ be a divisor. Then, 
\begin{enumerate}
    \item $c$ is maximal if and only if $\deg(c) = g$.
    \item $D$ is maximal winnable if and only if its q-reduced form is $c-q$, with maximal superstable $c$.
\end{enumerate}
\end{proposition}

\begin{proposition} (\cite[Corollary 4.9(3,4)]{corry2018divisors})\label{q-red-bij} \label{deg-max-unwinn} Let $D$ be a divisor on $G$.
\begin{enumerate}
    \item The correspondence \(\mathcal{O} \mapsto D(\mathcal{O}) \) is a bijection between acyclic orientations of $G$ with unique source $q$ and maximal unwinnable q-reduced divisors of $G$.
    \item If $D$ is a maximal unwinnable divisor, then $\deg(D) = g - 1$. Thus, $\deg(D) \geq g$ implies $D$ is winnable.
\end{enumerate}
\end{proposition} 

\subsection{Rank}

One of our winnability questions pertained to \textit{``Are some games more winnable than others?''} One way to answer this is by defining the \textit{(Baker--Norine) rank function}, which is central to proving the Riemann-Roch theorem for graphs and will lend us the ability to formalize further the answer to the question \textit{``How many chips can be removed for the divisor to remain still winnable?''}

\definition \label{rank-def} The \textit{rank function} $r(D)\colon D \to \mathbb{Z}$ is defined as:
\begin{enumerate}
    \item $r(D) = -1$ if and only if $|D| = \emptyset$.
    \item $r(D) \geq k$ for $k \geq 0$ if and only if the dollar game is winnable starting from all divisors obtained from $D$ by removing $k$ dollars.
    \item $r(D) = k$ if and only if $r(D) \geq k$ and there exists an effective divisor $E$ such that $\deg(E) = k + 1$ and $D - E$ is unwinnable.
\end{enumerate}

It is worth restating that the problem of computing the rank of a general divisor on a general graph is \textbf{NP}-hard \cite{kiss2015chip}, which means the time it takes for an algorithm to compute the rank is non-polynomial. Specifically, this time grows exponentially with the size of the graph \cite[\S 5.1]{corry2018divisors}. So while rank can be computed in small or special cases, one should not expect our package's rank function to scale well to large graphs in general.

\begin{corollary}
\label{rank-ineq}
For divisors $D, D'$ with $r(D),r(D') \geq 0$, we have $r(D+D') \geq r(D) + r(D')$.
\end{corollary}


\subsection{Riemann-Roch Theorem for Graphs}

\begin{theorem} \label{riemann-roch-graphs} (Riemann-Roch for graphs \cite{baker2007riemann}). Let $D$ be a divisor on a (loopless, undirected) graph $G$ of genus $g = |E| - |V| + 1$ with canonical divisor $K$. Then,
\[
r(D) - r(K - D) = 1 + \deg(D) - g.
\]
\end{theorem}

The reader new to this story may enjoy trying to prove the following consequences of Riemann--Roch. 

\begin{corollary}
\label{cor:cannonical-max-unwinn}
    A divisor $D$ is maximal unwinnable if and only if the divisor $K - D$ is maximal unwinnable.
\end{corollary}

\begin{theorem}[{Clifford's Theorem, \cite[Corollary 3.5]{baker2007riemann} or \cite[Corollary 5.13]{corry2018divisors}}]
\label{clifford-theorem}
Suppose $D \in \textrm{Div}(G)$ is a divisor with $r(D) \geq 0$ and $r(K - D) \geq 0$, then we have $r(D) \leq \frac{1}{2}\deg(D)$.
\end{theorem}

We remark that, while Corollary \ref{cor:cannonical-max-unwinn} is an easy consequence of Riemann--Roch, in a sense that says the story backwards. Indeed, the original proof of Riemann--Roch in \cite{baker2007riemann} deduces it as a consequence of this Corollary, which is first demonstrated with the aid of orientations; see the bottom of p. 15 of \cite{baker2007riemann}.

\begin{corollary}
\label{cor:ranks-determination}
Let $D \in \textrm{Div}(G)$.
\begin{enumerate}
\item If $\deg(D) < 0$, then $r(D) = -1$.
\item If $0 \leq \deg(D) \leq 2g - 2$, then $r(D) \leq \frac{1}{2}\deg(D)$.
\item If $\deg(D) > 2g - 2$, then $r(D) = \deg(D) - g$.
\end{enumerate}
\end{corollary}

\subsection{Computing rank in the package}

Here's an illustration of our \texttt{chipfiring} package API for efficient rank calculation:

\begin{lstlisting}[language=Python]
from chipfiring import rank

# Compute the rank with optional optimizations
# Setting optimized=True enables the use of theoretical results (e.g., Corollary 4.23)
# to short-circuit expensive EWD calls when a shortcut applies.
# Note: While faster, this mode may not return intermediate structures 
# like q-reduced divisors or orientations.
result = rank(divisor, optimized=True)

# Access the rank and logs
print("Rank:", result.rank)
print(result.get_log_summary())
\end{lstlisting}

\subsection{Gonality}

An interesting question arising from ranks of divisors is to determine the \emph{gonality} of a given finite graph, or specific families of graphs. We briefly state the question here and show how to study it in our package. An excellent introduction to this subject is \cite{beougher2023chip}, from which we have chosen some worked examples below.

\begin{definition}
The \emph{gonality} of a graph $G$ is the minimal degree of a divisor $D$ of rank at least $1$.
\end{definition}

An important open problem, motivated by algebraic geometry, is Baker's gonality conjecure \cite[Conjecture 3.10(1)]{baker2008}, which states that every graph of genus $g$ has gonality at most $\lfloor \frac12 g + \frac32 \rfloor$.
Gonality is difficult to compute in general, especially for large graphs, but for small graphs it can be determined by exhaustive search. Such a search is implemented in \texttt{chipfiring.CFGonality}. For example, our package confirms the gonalities of platonic solid graphs, as reported in \cite{beougher2023chip}. While this search take less than a few seconds for the tetrahedron, cube, and octahedron, the dodecahedron and tetrahedron take considerably more time, due to the number of divisors that must be considered. For example, the computation below took 10 minutes and 24 seconds on a Macbook Pro M2.

\begin{lstlisting}[language=Python]
# Compute the gonalities of Platonic solid graphs
from chipfiring import tetrahedron, cube, octahedron,  dodecahedron, icosahedron, gonality

# Platonic solid graphs (available for research purposes)
tetra = tetrahedron()  # K_4
cube_graph = cube()    # 3-regular, 8 vertices
octa = octahedron()    # Complete tripartite K_{2,2,2}
dodec = dodecahedron()  # 3-regular, 20 vertices
icos = icosahedron() # 5-regular, 12 vertices

print(gonality(tetra).gonality) # Output: 3
print(gonality(cube_graph).gonality) # Output: 4
print(gonality(octa).gonality) # Output: 4
print(gonality(dodec).gonality) # Output: 6
print(gonality(icos).gonality) # Output: 9
\end{lstlisting}

The reason that the word ``gonality'' must be written twice in each in the computations above is that \texttt{gonality(G)} returns an object that records not just the gonality computed, but a log of the computation and an example of a rank-$1$ divisor of minimum degree. For example, the command $\mathtt{D = gonality(tetra).winning\_strategies[0]}$ will give a degree $3$ divisor of rank $1$ on the tetrahedron graph.

While the modules (\texttt{CFCombinatorics}, \texttt{CFGonality}, \texttt{CFGonalityDhar}, \texttt{CFPlatonicSolids}) are functional, they may be subject to future modifications as the underlying theory and implementation continue to develop.

\subsection{An extended example: gonalities of chains of loops}

As an example of a more intensive computation that our package makes possible, we give below a longer example of a graph that has attracted consderable interested in divisor theory of graphs: the chain of loops (or chains of cycles). Such graphs first gained attention in \cite{cools2012tropical}, where it was shown that, if chosen sufficiently generically, they are ``Brill--Noether general;'' among other this means that their gonality is exactly $\lfloor \frac12 g + \frac32$. Later, the third author developed tools in \cite{pflChains} that suffice to determine ranks of all divisors on all chains of loops. Later, Jensen and Sawyer \cite{jensenSawyer} used these tools to classify the gonalities of all chains of cycles. 

The code below shows how to use the \texttt{chipfiring} package to verify the Jensen-Sawyer classificaiton of gonalities of chains of cycles. For simplicity and to limit the number of cases, we investigate the following special case.

\begin{example}
Consider a genus $5$ chain of cycles of the following form. $G$ consists of $5$ cycles, where the $n$th cycle consists of $m_n$ vertices $z_{n,0}, \cdots, z_{n,m_n-1}$ arranged in a cycle, and there is a single edge from each cycle to the next, joining $z_{n,0}$ to $z_{n+1,m_{n+1}-1}$. The numbers $m_1, \cdots m_5$ are called the \emph{cycle lengths} or \emph{torsion orders}. According to the classification in \cite{jensenSawyer}, the gonality of this graph is 
\[ \operatorname{gon}(G) = \begin{cases}
2 & \mbox{ if } m_2 = m_3 = m_4 = 2,\\
3 & \mbox{ if } \operatorname{gon}(G) \neq 2 \mbox{ and } m_2 =2, m_3 = 3, \mbox{ or } m_4 = 2,\\
4 & \mbox{ otherwise.}
\end{cases}\]
\end{example}

The following code demonstrates how to use our package to build a chain of cycles of this kind, compute its gonaltiy, and verify that it accords with the classification of Jensen and Sawyer. In order to reduce to finitely many cases, we limit all cycle lengths to $2,3,4,$ or $5$.

\begin{lstlisting}[language=Python]
from chipfiring import CFGraph
from chipfiring import gonality
import itertools

def chain(cycle_lengths : list[int]):
    if not all(isinstance(l, int) and l >= 2 for l in cycle_lengths):
        raise ValueError("All cycle lengths must be integers greater than or equal to 2.")
    vertices = { f"z_{i+1}_{j}"  for i,l in enumerate(cycle_lengths) for j in range(l)}
    # edges = [(f"z_{i+1}_{j}",f"z_{i+1}_{(j+1)%l}",1) for i,l in enumerate(cycle_lengths) for j in range(l)]
    edges = []
    for i,l in enumerate(cycle_lengths):
        if l == 2:
            edges.append((f"z_{i+1}_0", f"z_{i+1}_1", 2))
        else:
            for j in range(l):
                edges.append((f"z_{i+1}_{j}", f"z_{i+1}_{(j+1)%l}", 1))
    for i,l in enumerate(cycle_lengths):
        if i == 0 : continue
        edges.append( (f"z_{i}_0", f"z_{i+1}_{l-1}",1) )
    return CFGraph(vertices,edges)

def expectedGonality(cycle_lengths : list[int]) -> int:
    if not all(isinstance(l, int) and l >= 2 for l in cycle_lengths):
        raise ValueError("All cycle lengths must be integers greater than or equal to 2.")
    if cycle_lengths[1] == 2 and cycle_lengths[2] == 2 and cycle_lengths[3] == 2:
        return 2
    elif cycle_lengths[1] == 2 or cycle_lengths[2] == 3 or cycle_lengths[3] == 2:
        return 3
    else:
        return 4

length_choices = [2,3,4,5]

for cycle_lengths in itertools.product(length_choices, repeat=5):
    cycle_lengths = list(cycle_lengths)
    G = chain(cycle_lengths)
    expected = expectedGonality(cycle_lengths)
    computed = gonality(G).gonality
    print(cycle_lengths, f"Expected: {expected}", f"Computed: {computed}",expected == computed)
\end{lstlisting}

As one would hope, the output of this code consists of $1024$ lines, one for each choice of cycle lengths, confirming that our package computed the same gonality as that predicted by theoretical means. This computation is expensive, since it requires a considerable exhaustive search. On a Macbook Pro M2, this computation was completed in 167 minutes and 14 seconds.

\section{Conclusion and Future Directions}
\label{sec:conclusion}

This paper introduces the \texttt{chipfiring} Python package, a dedicated computational toolkit designed explicitly for the mathematical analysis of chip-firing games on multigraphs. We have developed specialized data structures, algorithms, and theoretical frameworks useful for exploration and experimentation within this domain.

Key contributions of the \texttt{chipfiring} package include specialized object-oriented classes such as \texttt{CFGraph}, \texttt{CFDivisor}, \texttt{CFLaplacian}, \texttt{CFOrientation}, \texttt{CFConfig}, \texttt{CFiringScript}, and \texttt{CFDataProcessor} tailored specifically for chip-firing dynamics. Algorithmically, we have implemented efficient methodologies such as Dhar's burning algorithm (\texttt{CFDhar}), the Greedy algorithm (\texttt{CFGreedyAlgorithm}), and rank calculation techniques (\texttt{CFRank}). These tools collectively also facilitate rapid and precise analyses of winnability (\texttt{EWD, is\_winnable}), linear equivalence (\texttt{linear\_equivalence}), and q-reduction (\texttt{q\_reduction, is\_q\_reduced}). Furthermore, comprehensive documentation, intuitive APIs, and interactive visualization modules (\texttt{CFVisualizer}) augment the educational and research applicability of the package, making it accessible to a broad audience.

\subsection{Current Scope and Limitations}

The current scope of the \texttt{chipfiring} package focuses primarily on finite, undirected multigraphs without loops, supporting fundamental mathematical operations pertinent to chip-firing theory. While extensive, our current implementation is optimized primarily for standard scenarios encountered frequently in graph theory, algebraic geometry, and combinatorial optimization contexts. Presently, the package’s computational complexity allows efficient analyses on medium-sized graphs but faces scalability challenges when handling extremely large-scale systems, highlighting an area ripe for future optimization.

Moreover, the package currently does not fully integrate advanced symbolic computation tools, limiting exact arithmetic operations. The visualization functionalities provided are adequate for typical educational and research needs but could be further enhanced to represent more complex and dynamic chip-firing processes clearly.

\subsection{Future Directions}

Several exciting avenues exist for future development and enhancement of the \texttt{chipfiring} package. One promising direction is the extension to more general classes of graphs, including directed, weighted, and loop-containing variants, which would greatly broaden the package's applicability and research scope. Furthermore, incorporating symbolic computation capabilities and sandpile models \cite{perkinson2011primeralgebraicgeometrysandpiles} could significantly enhance mathematical rigor and exactness, particularly valuable in theoretical explorations requiring precise arithmetic operations.

Additionally, performance optimization using parallel computing techniques, GPU acceleration, and efficient sparse matrix computations would greatly enhance scalability and facilitate the handling of larger and more complex graphs.

On the theoretical front, integrating more extensive modules supporting ongoing research in gonality computation, divisor specialization, Brill-Noether theory \cite{cools2012tropical, pflueger2017brill}, and advanced combinatorial bounds \cite{osserman2017limit} will keep the package relevant to current research topics.

\subsection{Community Engagement and Availability}\label{sec:community}

The \texttt{chipfiring} package is actively maintained and openly available to the mathematical and computational community. We strongly encourage community engagement through feedback, contributions, and feature requests via the project's GitHub repository. This package is licensed under the permissive MIT License.

\section{Acknowledgments}

We thank Ralph Morrison for feedback and guidance. Special thanks to the Amherst College Department of Mathematics for supporting this research project. We also acknowledge the influence of existing mathematical software packages (i.e. \texttt{numpy, networkx, matplotlib, pandas, dash}) that inspired our approach to specialized scientific computing tools. Portions of the code were developed with the assistance of AI coding tools, including Cursor with Gemini 2.5 Pro, Claude 3.5 Sonnet, and Claude 3.7 Sonnet.

\bibliographystyle{unsrt}
\bibliography{references} 

\newpage
\begin{appendices}
\section{Data Import Format} \label{appendix:import}

\subsection{JSON Import Formats}

\subsubsection{Graph (\texttt{CFGraph})}
\begin{verbatim}
{
    "vertices": ["vertex1", "vertex2", "vertex3"],
    "edges": [
        ["vertex1", "vertex2", valence],
        ["vertex2", "vertex3", valence]
    ]
}
\end{verbatim}
\begin{itemize}
    \item \textbf{vertices}: Array of vertex names (strings)
    \item \textbf{edges}: Array of [source, target, valence] tuples where valence is a positive integer
\end{itemize}

\subsubsection{Divisor (\texttt{CFDivisor})}
\begin{verbatim}
{
    "graph": {
        "vertices": ["vertex1", "vertex2"],
        "edges": [["vertex1", "vertex2", 1]]
    },
    "degrees": {
        "vertex1": 2,
        "vertex2": -1
    }
}
\end{verbatim}
\begin{itemize}
    \item \textbf{graph}: Graph object with vertices and edges structure
    \item \textbf{degrees}: Object mapping vertex names to integer degrees (can be negative)
\end{itemize}

\subsubsection{Orientation (\texttt{CFOrientation})}
\begin{verbatim}
{
    "graph": {
        "vertices": ["vertex1", "vertex2"],
        "edges": [["vertex1", "vertex2", 1]]
    },
    "orientations": [
        ["source", "sink"],
        ["vertex1", "vertex2"]
    ]
}
\end{verbatim}
\begin{itemize}
    \item \textbf{graph}: Graph object with vertices and edges structure
    \item \textbf{orientations}: Array of [source, sink] pairs defining edge directions
\end{itemize}

\newpage

\subsubsection{Firing Script (\texttt{CFiringScript})}
\begin{verbatim}
{
    "graph": {
        "vertices": ["vertex1", "vertex2"],
        "edges": [["vertex1", "vertex2", 1]]
    },
    "script": {
        "vertex1": 3,
        "vertex2": 0
    }
}
\end{verbatim}
\begin{itemize}
    \item \textbf{graph}: Graph object with vertices and edges structure
    \item \textbf{script}: Object mapping vertex names to firing counts (non-negative integers)
\end{itemize}

\subsection{TXT Import Formats}

\subsubsection{Graph (\texttt{CFGraph})}
\begin{verbatim}
VERTICES: vertex1, vertex2, vertex3
EDGE: vertex1, vertex2, valence
EDGE: vertex2, vertex3, valence
\end{verbatim}
\begin{itemize}
    \item \textbf{VERTICES}: Comma-separated list of vertex names
    \item \textbf{EDGE}: Format \texttt{source, target, valence} where valence is a positive integer
\end{itemize}

\subsubsection{Divisor (\texttt{CFDivisor})}
\begin{verbatim}
GRAPH_VERTICES: vertex1, vertex2, vertex3
GRAPH_EDGE: vertex1, vertex2, valence
GRAPH_EDGE: vertex2, vertex3, valence
---DEGREES---
DEGREE: vertex1, degree_value
DEGREE: vertex2, degree_value
\end{verbatim}
\begin{itemize}
    \item \textbf{GRAPH\_VERTICES}: Comma-separated list of vertex names
    \item \textbf{GRAPH\_EDGE}: Format \texttt{source, target, valence}
    \item \textbf{---DEGREES---}: Separator line
    \item \textbf{DEGREE}: Format \texttt{vertex\_name, degree\_value} (degree can be negative)
\end{itemize}

\subsubsection{Orientation (\texttt{CFOrientation})}
\begin{verbatim}
GRAPH_VERTICES: vertex1, vertex2, vertex3
GRAPH_EDGE: vertex1, vertex2, valence
GRAPH_EDGE: vertex2, vertex3, valence
---ORIENTATIONS---
ORIENTED: source, sink
ORIENTED: vertex1, vertex2
\end{verbatim}
\begin{itemize}
    \item \textbf{GRAPH\_VERTICES}: Comma-separated list of vertex names
    \item \textbf{GRAPH\_EDGE}: Format \texttt{source, target, valence}
    \item \textbf{---ORIENTATIONS---}: Separator line
    \item \textbf{ORIENTED}: Format \texttt{source, sink} defining edge direction
\end{itemize}

\subsubsection{Firing Script (\texttt{CFiringScript})}
\begin{verbatim}
GRAPH_VERTICES: vertex1, vertex2, vertex3
GRAPH_EDGE: vertex1, vertex2, valence
GRAPH_EDGE: vertex2, vertex3, valence
---SCRIPT---
FIRING: vertex_name, firing_count
FIRING: vertex1, 3
\end{verbatim}
\begin{itemize}
    \item \textbf{GRAPH\_VERTICES}: Comma-separated list of vertex names
    \item \textbf{GRAPH\_EDGE}: Format \texttt{source, target, valence}
    \item \textbf{---SCRIPT---}: Separator line
    \item \textbf{FIRING}: Format \texttt{vertex\_name, firing\_count} (non-negative integer)
\end{itemize}
\end{appendices}

\end{document}